\documentclass{article}

\usepackage{PRIMEarxiv}

\usepackage{amsthm}

\newtheorem{theorem}{Theorem}

\newtheorem{corollary}{Corollary}[theorem]
\newtheorem{proposition}{Proposition}[theorem]
\newtheorem{lemma}{Lemma}[theorem]
\newtheorem{definition}{Definition}[theorem]





\usepackage[utf8]{inputenc} % allow utf-8 input
\usepackage[T1]{fontenc}    % use 8-bit T1 fonts
\usepackage{hyperref}       % hyperlinks
\usepackage{url}            % simple URL typesetting
\usepackage{booktabs}       % professional-quality tables
\usepackage{nicefrac}       % compact symbols for 1/2, etc.
\usepackage{microtype}      % microtypography
\usepackage{lipsum}
\usepackage{fancyhdr}       % header
\usepackage{graphicx}       % graphics
\graphicspath{{media/}}     % organize your images and other figures under media/ folder

\usepackage{amsmath,amssymb,amsfonts}

\usepackage{tikz}
\usepackage{natbib}
 \bibpunct[, ]{(}{)}{,}{a}{}{,}%

\usepackage{rotating}
\usepackage{fancyvrb}
\usepackage{subcaption}
\usepackage[font=small]{caption}
\usepackage{optidef}
\usepackage{algorithm}
\usepackage{algorithmic}
\usepackage{longtable}
\usepackage{multirow}
\usepackage{accents}

\newcommand{\ubar}[1]{\underaccent{\bar}{#1}}
\newcommand*{\probability}{\mathbb{P}}
\newcommand*{\ev}{\mathbb{E}}
\newcommand*{\compon}{k}
\newcommand*{\scen}{s}

\newcommand*{\alloclevel}{\ell}
\newcommand*{\componoutcome}{\kappa}
\newcommand*{\cell}{\lambda}
\newcommand*{\cellalt}{\cell'}
\newcommand*{\leaf}{\lambda}

\newcommand*{\secondstagecost}{q}
\newcommand*{\secondstagecostvec}{\textbf{\secondstagecost}}
\newcommand*{\capac}[1][]{u_{#1}}

\newcommand*{\numcompon}{n}
\newcommand*{\numalloclevels}{L}
\newcommand*{\numscens}{S}

\newcommand*{\numfirststagevars}{n_0}

\newcommand*{\numsecondstagevars}{n_2}
\newcommand*{\numsecondstagemainconstrs}{m}
\newcommand*{\numcells}{\Lambda}

\newcommand*{\firststagevar}{x}

\newcommand*{\firststagevec}{\textbf{\firststagevar}}

\newcommand*{\firststagevecfix}{\hat{\textbf{\firststagevar}}}

\newcommand*{\firststageallocvar}{x}
\newcommand*{\firststagealloc}[1][]{\firststageallocvar_{#1}}

\newcommand*{\firststageallocil}{\firststagealloc[\compon\alloclevel]}

\newcommand*{\firststageallocvec}{\textbf{\firststageallocvar}}
\newcommand*{\firststageallocvecfix}{\hat{\textbf{\firststageallocvar}}}

\newcommand*{\secondstagevar}{y}
\newcommand*{\secondstagevec}{\textbf{\secondstagevar}}

\newcommand*{\auxvar}{\theta}

\newcommand*{\objvar}{z}
\newcommand*{\objvarlb}{\ubar{\objvar}}

\newcommand*{\capdualvar}{\beta}

\newcommand*{\randsym}{\xi}

\newcommand*{\randvark}[1]{\randsym_{#1}}
\newcommand*{\randvec}{\boldsymbol{\randsym}}
\newcommand*{\randvecsupport}{\Xi}

\newcommand*{\partition}{\Phi}
\newcommand*{\partitioncellvar}{\randvecsupport}
\newcommand*{\partitioncell}[1][]{\partitioncellvar^{#1}}
\newcommand*{\partitioncellk}{\partitioncell[\cell]}
\newcommand*{\evcell}[1]{\bar{#1}^{\partitioncell}}
\newcommand*{\evcellindex}[1]{\bar{#1}^{\partitioncell[\cell]}}
\newcommand*{\randvarevfn}[1]{\bar{\randsym}_\compon(#1)}
\newcommand*{\randvecevfn}[1]{\bar{\randvec}(#1)}

\newcommand*{\randvecevcellfn}[1]{\evcell{\randvec}(#1)}
\newcommand*{\randvecevcellfnindex}[1]{\evcellindex{\randvec}(#1)}
\newcommand*{\randvecevcell}{\randvecevcellfn{\uncertaintydrivervec}}
\newcommand*{\recourseevcellfn}[1]{\overline{\recfn(\randvec)}^{\partitioncell}(#1)}
\newcommand*{\treenode}{\cell}

\newcommand*{\recfn}{Q}
\newcommand*{\recfnpenalty}{\tilde{\recfn}}
\newcommand*{\exprecfn}{\mathcal{\recfn}}

\newcommand*{\probfnvar}{\probability}

\newcommand*{\condprobfn}[2]{\probfnvar[#1\,|\,#2]}
\newcommand*{\stateprobfn}[2]{f_{#1}(#2)}

\newcommand*{\stateprobfndiscri}{\stateprobfn{\compon}{\staterealizi;\alloclevel}}

\newcommand*{\stateprobfndiscris}{\stateprobfn{\compon}{\staterealizis;\alloclevel}}

\newcommand*{\evfnrv}[2]{\ev_{#1}[#2]}
\newcommand*{\condEV}[2]{\ev[#1\,|\,#2]}
\newcommand*{\probcellvar}{\probability^{\partitioncell}}
\newcommand*{\probcellvarindex}{\probability^{\partitioncell[\cell]}}
\newcommand*{\probfncell}[1]{\probcellvar[#1]}
\newcommand*{\probfncellindex}[1]{\probcellvarindex[#1]}
\newcommand*{\condprobfncell}[2]{\probcellvar[#1\,|\,#2]}

\newcommand*{\firststagefeasregion}{X}

\newcommand*{\firststageallocfeasregion}{X}
\newcommand*{\recoursefeasregion}{Y}

\newcommand*{\components}{I}
\newcommand*{\fixedcomponents}{F}

\newcommand*{\descendentnodes}{\mathcal{D}}
\newcommand*{\parttreenodes}{\mathcal{N}}
\newcommand*{\parttreeleafnodes}{\mathcal{L}}
\newcommand*{\arcs}{A}

\newcommand*{\alloclevelssum}{\sum_{\alloclevel=0}^\numalloclevels}

\newcommand*{\probprod}{\prod_{\compon=1}^\numcompon}
\newcommand*{\scenssum}{\sum_{s=1}^\numscens}

\newcommand*{\cellssum}{\sum_{\partitioncell\in\partition}}
\newcommand*{\constrdualvar}{\alpha}
\newcommand*{\constrdual}[1][]{\constrdualvar_{#1}}
\newcommand*{\constrdualvec}{\boldsymbol{\constrdual}}

\newcommand*{\jointsupport}{\Delta}
\renewcommand*{\partitioncellvar}{\jointsupport}

\newcommand*{\condtreeprob}[1]{\probability^{\treenode',\treenode}[#1]}

\newcommand*{\attackalloclevel}{\ell'}
\newcommand*{\attackplan}{\kappa}

\newcommand*{\secondstagerhsfixed}{b}
\newcommand*{\secondstagerhsfixedvec}{\textbf{\secondstagerhsfixed}}
\newcommand*{\secondstagerandvarcoef}{d}
\newcommand*{\secondstagerandvarcoefvec}{\textbf{\secondstagerandvarcoef}}

\newcommand*{\numattackalloclevels}{L'}
\newcommand*{\numattackplans}{\bar{\attackplan}}
\newcommand*{\numattackvars}{n_1}

\newcommand*{\firststageallocfixed}[1][]{\hat{\firststageallocvar}_{#1}}
\newcommand*{\firststageallocfixedil}{\firststageallocfixed[\compon\alloclevel]}

\newcommand*{\attackvar}{v}

\newcommand*{\attackvec}{\textbf{\attackvar}}

\newcommand*{\attackvecfix}{\hat{\textbf{\attackvar}}}

\newcommand*{\attackallocvar}{v}
\newcommand*{\attackalloc}[1][]{\attackallocvar_{#1}}
\newcommand*{\attackallocfixed}[1][]{\hat{\attackallocvar}_{#1}}

\newcommand*{\attackallocil}{\attackalloc[\compon\attackalloclevel]}
\newcommand*{\attackallocfixedil}{\attackallocfixed[\compon\attackalloclevel]}

\newcommand*{\constrdualvarub}{\beta}
\newcommand*{\constrdualub}[1][]{\constrdualvarub_{#1}}
\newcommand*{\constrdualubvec}{\boldsymbol{\constrdualub}}

\newcommand*{\attackobjvar}{\zeta}
\newcommand*{\attackobjvarlb}{\ubar{\attackobjvar}}

\renewcommand*{\stateprobfndiscri}{\stateprobfn{\compon}{\tilde{\randsym}_\compon;\alloclevel,\attackalloclevel}}
\renewcommand*{\stateprobfndiscris}{\stateprobfn{\compon}{\tilde{\randsym}_\compon^\scen;\alloclevel,\attackalloclevel}}

\newcommand*{\attackfeasregion}{V}
\newcommand*{\attackplansset}{\bar{\attackfeasregion}}

\renewcommand*{\randvecevcell}{\randvecevcellfn{\firststagevec,\attackvec}}
\newcommand*{\recourseevcell}{\recourseevcellfn{\firststagevec,\attackvec}}

\newcommand*{\secondstagematrixfixed}{A} % override

\newcommand*{\attackplansinlist}{\attackplan=1,\dots,\numattackplans}

\newcommand*{\attackalloclevelssum}{\sum_{\attackalloclevel=0}^{\numattackalloclevels}}
\usepackage[english]{babel}

\title{A Successive Refinement Algorithm for Tri-Level Stochastic Defender-Attacker Problems with Decision-Dependent Probability Distributions
\thanks{\textit{\underline{Citation}}: 
\textbf{Authors. Title. Pages.... DOI:000000/11111.}} 
}

\author{
  Samuel Affar, Hugh Medal \\
  Department of Industrial Engineering\\
  University of Tennessee \\
  Knoxville\\
  \texttt{saffar@vols.utk.edu, hmedal@utk.edu} \\
}

\begin{document}
\maketitle

\begin{abstract}
Tri-level defender-attacker game models are a well-studied method for determining how best to protect a system (e.g., a transportation network) from attacks. Existing models assume that defender and attacker actions have a perfect effect, i.e., system components hardened by a defender cannot be destroyed by the attacker, and attacked components always fail. Because of these assumptions, these models produce solutions in which defended components are never attacked, a result that may not be realistic in some contexts. 
% Methodology and Results
This paper considers an imperfect defender-attacker problem in which defender decisions (e.g., hardening) and attacker decisions (e.g., interdiction) have an imperfect effect such that the probability distribution of a component's capacity depends on the amount of defense and attack resource allocated to the component. Thus, this problem is a stochastic optimization problem with decision-dependent probabilities and is challenging to solve because the deterministic equivalent formulation has many high-degree multilinear terms. To address the challenges in solving this problem, we propose a successive refinement algorithm that dynamically refines the support of the random variables as needed, leveraging the fact that a less-refined support has fewer scenarios and multilinear terms and is, therefore, easier to solve. A comparison of the successive refinement algorithm versus the deterministic equivalent formulation on a tri-level stochastic maximum flow problem indicates that the proposed method solves many more problem instances and is up to $66$ times faster. 
% importance and implications
These results indicate that it is now possible to solve tri-level problems with imperfect hardening and attacks.
\end{abstract}

% keywords can be removed

\section{Introduction}
Many decisions in defense and homeland security involve \textit{adversarial resource allocation} in which an agent must determine how to allocate resources given that an adversary observes the allocation and allocates their own resources accordingly. To address this important class of problems, there is a mature body of literature on \textit{network interdiction}, which models a two-player game between an \textit{interdictor} that interdicts components of a network and an \textit{operator} that utilizes the remaining components to optimize some objective. From this literature emerged another body of literature on hardening (a.k.a., fortification) problems in which a \textit{defender}, acting prior to the interdictor, hardens some of the network components, making them immune to interdiction.

Most hardening problems assume \textit{perfect} hardening and interdiction such that a hardened component cannot fail and an interdicted component is guaranteed to fail. For many applications this is a reasonable assumption. For instance, when the effect of hardening is strong, the interdictor is not likely to attempt to attack a hardened component because they can do better by attacking unhardened components.

However, it is not difficult to identify applications of hardening models for which hardening and interdiction are not perfect. A prominent example is a conflict between two adversaries in which one allocates combat units to arcs in a network maximize the flow on a network and another allocates units in order to minimize flow. In this context, it is usually not reasonable to assume that a defended arc is impenetrable. Rather, it is more reasonable to assume that the outcome of a particular battle at an arc may depend on how many units are allocated to that arc. For applications such as this with \textit{imperfect} hardening and interdiction, an important question is: \textit{do existing hardening models (which assume perfect hardening and interdiction) still provide good-quality solutions to problems with imperfect hardening and interdiction}? The following example indicates that existing models can output poor solutions for these problems.

\subsection{Example}
%The following example illustrates that it is important to model imperfect interdiction and hardening. 
Consider the problem of hardening arcs in a network (shown in Figure \ref{fig:hardening-determ}) to mitigate the effects of interdiction and maximize post-interdiction maximum flow. The defender and attacker have a budget of $6$ to allocate among arcs in the network; each arc can be allocated $0$, $1$, or $2$ hardening (attack) resources.

A standard deterministic tri-level model was used to obtain an approximate solution to the problem. This standard model assumes \textit{binary} hardening and interdiction, which translates to each arc receiving either $0$ or $2$ hardening (interdiction) resources; both the defender and attacker have a budget of 6 units. This model also assumes \textit{perfect} hardening, i.e., a hardened arc cannot fail, and an interdicted arc fails with probability $1$. In this solution (shown in Figure \ref{fig:hardening-determ}), arcs $(2,3)$, $(7,8)$ and $(8,9)$ are hardened (with 2 units). Assuming that hardening and attacking arcs have a perfect effect, a best response for the attacker is to remove arcs $(1,2)$, $(2,5)$, and $(4,5)$, resulting in a maximum flow of $30$. However, under imperfect hardening and interdiction, this solution may not be optimal.  

Using the defender's solution $(2,3)$, $(7,8)$ and $(8,9)$, we solve an interdiction problem that assumes imperfect attack, i.e., the probability that an arc fails depends on the number of defender and attacker units allocated to the arc. In the solution from this approach, the attacker's optimal allocation is more dispersed, allocating one unit to arcs $(1,2)$, $(4,5)$, $(7,8)$, and $(8,5)$ and two units to arc $(8,9)$. As a result, the unhardened arcs fail with probability $1$, arc $(7,8)$ has a $1/3$ chance of failure, and arc $(8,9)$ a $1/2$ chance, resulting in an expected post-interdiction maximum flow of $10$, which is much less than predicted by the deterministic model. Thus, we observe that the deterministic model's inability to account for imperfect hardening and interdiction yields a hardening solution with a concentrated resource allocation, which leaves arcs $(1,2)$, $(4,5)$, and $(8,5)$ unprotected.

Figure \ref{fig:hardening-multAlloc} shows the hardening solution obtained by solving the true problem that considers imperfect hardening and attacks. In this solution, the hardening allocation is more dispersed, with one unit allocated to arcs $(1,2)$, $(4,7)$, and $(2,3)$ and $(7,8)$ and two units allocated to $(8,9)$, leaving fewer arcs unprotected with no hardening resources. In response to this more dispersed hardening plan, the attacker's solution (obtained by solving an interdiction model that assumes imperfect attacks) allocates one unit to arcs ($5,6)$ and ($7,8)$ and two units to $(2,3)$ and $(8,9)$, resulting in an expected post-interdiction maximum flow of $59/3\approx 19.66$, an increase of about 96\% over the solution produced by the standard model. 
%(Note that this solution is more concentrated than the best attack response to the deterministic hardening strategy (shown in Figure \ref{fig:hardening-determ}). 

As this example illustrates, when hardening and interdiction are not perfect, it may be unwise to assume that it is. However, as shown in \S\ref{sec:background} the models and algorithms that currently exist in the literature are not well-suited for computing a solution that performs well under imperfect hardening and interdiction. Thus, goal of this paper is to identify a tractable approach to solving tri-level problems with imperfect hardening and attacks, enabling the computation of better solutions.

\begin{figure}[ht]
     \centering
     %\begin{subfigure}[t]{0.45\textwidth}
         \centering
         \includegraphics[width=0.5\textwidth]{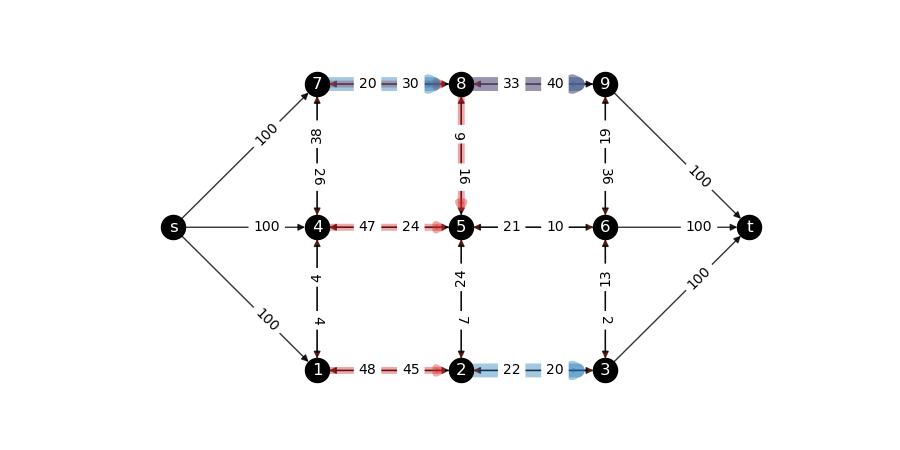}
         \caption{\textbf{Hardening solution obtained by deterministic tri-level model}. A standard deterministic tri-level model was used to (heuristically) obtain a strategy for allocating hardening units to arcs. This strategy yields an expected post-interdiction maximum flow of $10$ when arcs fail probabilistically.}
         \label{fig:hardening-determ}
     %\end{subfigure}
\end{figure}

\begin{figure}[ht]
     %\begin{subfigure}[t]{0.45\textwidth}
         \centering
         \includegraphics[width=0.5\textwidth]{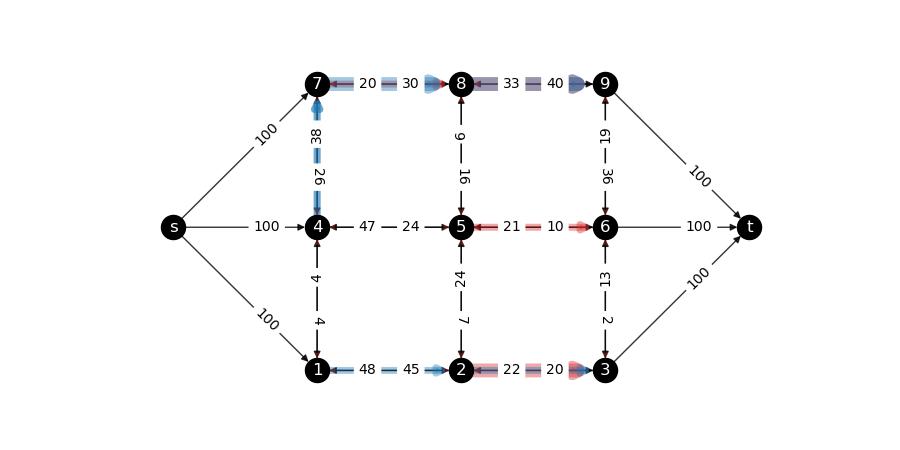}
         
     %\end{subfigure}
    %\caption{Comparison of perfect and imperfect hardening and interdiction for a maximum flow problem. Both the attacker and defender have $6$ resources.}
    \caption{\textbf{Solution obtained by model with imperfect hardening and interdiction}. Under this model the hardening solution is more dispersed, spanning $5$ arcs, which results in a different attacker solution, resulting in an expected post-interdiction maximum flow of $59/3\approx 19.66$.}
         \label{fig:hardening-multAlloc}
    \label{fig:illustration}
\end{figure}

\subsection{Contributions and Findings}
This paper contributes the following to the literature.
First, we identify \textit{lower and upper bound properties} inherent in a class of tri-level stochastic defender-attacker-operator problems with decision-dependent probabilities (\S\ref{sec:struct-properties}). Second, we leverage these properties to develop a \textit{stochastic refinement framework} for tri-level stochastic defender-attacker-operator with decision-dependent probabilities (\S\ref{sec:methods}).

Third, we then \textit{apply this framework to a class of tri-level stochastic defender-attacker problems with decision-dependent capacity distributions}, assuming that the random variables are discrete (with finite support) and statistically independent. Specifically, we reformulate the problem using duality, and describe an approach that \textit{dynamically refines the support of the random variables within a branch-and-bound algorithm} (\S\ref{sec:application}). This algorithm is demonstrated on a \textit{numerical example} of a tri-level defender-attacker maximum flow problem in which the defender and attacker allocate resources to arcs in order to maximize (minimize) the resulting expected maximum flow (\S\ref{sec:example}).

Fourth, we test the algorithm on a set of instances of the same tri-level defender-attacker maximum expected flow problem, comparing it with a deterministic equivalent formulation. Results indicate that \textit{the SRA algorithm is able to solve many more instances of the problem, especially larger ones}. Of the instances solved by both methods, \textit{the SRA algorithm is up to 66 times faster} (\S\ref{sec:results}).

\subsection{Outline of the paper}
In the remainder of this article, \S\ref{sec:background} discusses previous studies that are relevant to this work. \S\ref{sec:problem} describes the tri-level defender-attacker-operator problem studied in this paper. \S\ref{sec:struct-properties} discusses several structural properties that exist in tri-level defender-attacker-operator problems with decision-dependent probabilities and \S\ref{sec:methods} leverages these properties in developing the successive refinement framework for defender-attacker problems with decision-dependent uncertainty. \S\ref{sec:application} then applies this framework to a class of defender-attacker interdiction problems with fortification and decision-dependent probabilities in which the random variables are discrete (with finite support), and mutually statistically independent. \S\ref{sec:example} demonstrates the algorithm on a tri-level maximum flow interdiction problem with decision-dependent capacity distributions.
\S\ref{sec:results} presents results that evaluate the scalability of the successive refinement algorithm, and \S\ref{sec:conclusion} concludes the article.

\section{Background}\label{sec:background}

The literature on network interdiction dates at least to \cite{mcmasters1970optimal}, and includes the notable paper on the maximum flow network interdiction problem by \cite{Wood.1993}, which led to many extensions. \cite{Cormican.1998} studied a stochastic network interdiction problem in which an attacked arc is completely destroyed according to some fixed probability. A extension is also discussed in which multiple attacks can be made on an arc, each successful with a fixed probability. This extension is similar to the interdiction problem we consider in this paper except that we assume a more general relationship between the number of attacks and probability that the arc is destroyed. 
%\cite{Janjarassuk.2008} showed that sample average approximation could be used effectively to solve the problem. 
\cite{Losada.2012} studied a multi-level uncapacitated facility interdiction problem in which the attacker chooses the interdiction level for each facility, with higher levels having a higher chance of success. 

A related concept to imperfect interdiction is \textit{partial} interdiction. \cite{aksen2014bilevel} integrated partial facility interdiction decisions into a median-type network interdiction problem with capacitated facilities and outsourcing option, representing the problem as a bi-level defender-attacker model. In this model, the interdictor can destroy any portion of the facility's capacities as far as it is within the budget. This study proposed two solution techniques: a progressive grid search and a multi-start simplex search heuristics. \cite{fard2018bi} developed two meta-heuristics to solve an extension of the problem where the defender aims to minimize two objective functions simultaneously. %\cite{forghani2020bi} developed a heuristic approach for the problem of interdicting nested hierarchical facilities. 
Some studies have considered \textit{partial} hardening. \cite{hien2020mitigating} developed a stochastic defender-attacker-defender (DAD) sequential game model to optimally allocate defensive resources to a transportation network where the operator hardens the network by selecting arcs to partially or entirely protect their nominal capacities or by reinforcing them with indestructible secondary capacities. 

There are only a few papers that have considered hardening problems in which interdictions are probabilistically successful. \cite{Zhu.2013} and \cite{Li.2021p8} studied a problem in which an attack on a facility is probabilistically successful, dependent on the number of defense units assigned to the facility, using heuristic approaches to solve the problems. However, there are currently no exact methods for hardening problems with imperfect hardening and interdiction. 

While the hardening literature generally does not consider an imperfect effect of allocating defense/attack resources, in the literature on \textit{security games} this is a basic assumption. Generally speaking, this literature focuses on simultaneous games and seeks to identify closed-form solutions for equilibrium decisions. \cite{hunt2024review} offers a review of Stackelberg security games, highlighting that most studies assume imperfect hardening and interdiction. \cite{xu2016modeling} %and \cite{hausken2011governments} 
studied the interaction between the government's defense systems and terrorist attacks using defense-attack games with imperfect hardening and attack measures. This body of literature differs from the present study in that security games generally do not include a recourse problem, and are solved analytically.

% A related area of study is hardening against random disruptions. \cite{Medal.2015} studied a multi-level capacitated facility protection problem with partial failures in which protecting a facility at a higher level increases the probability that facility will have more capacity remaining after a random disruption. 

The imperfect hardening problem can be described as a stochastic optimization problem with \textit{decision-dependent uncertainty} (DDU) due to the fact that the uncertainty in the problem (i.e., which attacks are successful) depends on the decisions (i.e., how to allocate hardening and interdiction resources). Stochastic optimization with DDU is an active area of research and currently most solution approaches are problem-specific. Research on stochastic optimization with DDU can be categorized based on how decisions influence uncertainty, as outlined by \cite{Goel.Grossmann.2006}. Decisions can impact the timing of uncertainty realization \citep{Goel.Mulkay.2006}, the uncertainty sets in robust and distributionally-robust optimization models \citep{Luo.Mehrotra.2020}, and the probability distribution of random variables \citep{Liu.Sen.2022}. Our study falls into the third category, the least explored in the literature, in which decisions influence the probabilities of random variables, i.e., decision-dependent probabilities. 
%We model component availability probabilities as a function of defense and attack decisions made in the first two stages. 

Several approaches have been developed to address decision-dependent probabilities. One method is to decouple the decision dependence in the objective before solving the model. This approach involves applying decoupling techniques to avoid non-convexity and using a sampling technique or a heuristic to solve the resulting model, which usually has a larger sample space. This method is common in power grid planning and distribution \citep{Ma.Guo.2017}.%Zhang.Cao.2023 
 The drawback is the increased computational complexity due to the larger scenario set. Another technique involves discretizing the first-stage variables and linearizing the multilinear terms using a probability chain. This method, used in \cite{Losada.2012,OHanley.2013}, and \cite{Medal.2015} simplifies the model but increases the size of the resulting mixed-integer programs (MIPs), hindering scalability. %\cite{Medal.2015} applied a similar approach for a two-stage facility location stochastic program with endogenous uncertainty.

Some studies use approximations to resolve decision dependence. For instance, the Point Estimate Method (PEM) \cite{Yin.Hou.2023} reduces the scenario sample space by using special point concentrations instead of the entire probability density function, approximating the expectation of the second-stage objective. \cite{Peeta.Viswanath.2010} approximated the objective function using a Taylor series expansion to resolve the non-convexity. While these methods reduce computational complexity, they may sacrifice solution accuracy. Additionally, heuristic approaches are employed to find locally optimal solutions. For example, \cite{karaesmen2004overbooking} used a stochastic gradient descent method for the stochastic overbooking problem, and \cite{Held.Woodruff.2005} and \cite{Du.Peeta.2014} used iterative heuristic algorithms with sampling for multi-stage stochastic network interdiction problems. These heuristic methods are less computationally intensive but do not guarantee globally optimal solutions.

\cite{medalaffar24} developed an algorithm based on the successive refinement framework that provides an exact solution to a specific class of stochastic programs with decision-dependent uncertainty, specifically stochastic programs with decision-dependent random capacities. This class of problems considered in \cite{medalaffar24} are single-stage problems (and two-stage problems that can be combined into a single stage) which differ from the tri-level defender-attacker problem studied in this paper.

\section{Problem Description}\label{sec:problem}
In this paper we seek to solve the following tri-level defender-attacker-operator formulation:
\begin{mini}|s|
    {\firststagevec\in \firststagefeasregion}{\max_{\attackvec\in \attackfeasregion}\quad  \exprecfn(\firststagevec,\attackvec),}{\label{mod:dad-exp-recourse-fn}}{\objvar^*=}
\end{mini}

\noindent where
$\firststagevec$ is a vector of first-level defense decisions made by a \textit{defender}. After observing the decision made by the defender, an \textit{attacker} makes the second-level decision represented by the vector $\attackvec$. The primary motivation for this paper is problems in which $\firststagevec$ the allocation of defense resources allocated in order to defend against an attack, and $\attackvec$ represents the allocation of attack resources intended to destroy components of a system. 
The vectors have feasible region $\firststagefeasregion\subseteq\mathbb{R}^{\numfirststagevars}$ and $\attackfeasregion\subseteq\mathbb{R}^{\numattackvars}$, respectively, each of which incorporate constraints such as resource scarcity.
The function $\exprecfn(\firststagevec,\attackvec)=\evfnrv{\randvec\,|\,\firststagevec,\attackvec}{\recfn(\randvec)}$ is the expected recourse function, with the expectation taken over the random vector $\randvec$ (with support $\randvecsupport\subseteq\mathbb{R}^\numcompon$) that is conditional on the defender and attacker decisions $\firststagevec$ and $\attackvec$, implying that this is a \textit{stochastic program with decision-dependent probabilities}. After the random vector $\randvec$ is realized, the recourse function $\recfn$ computes the objective value for the third-level operator's problem:

\begin{mini!}|s|
    {\secondstagevec\in \recoursefeasregion}{\secondstagecostvec^T\secondstagevec \label{eq:obj-recourse}}{\label{mod:recourse}}{\recfn(\randvec) =}
    \addConstraint{\secondstagematrixfixed\secondstagevec}{= \secondstagerhsfixedvec}{\quad [\constrdualvec]\label{constr:fixed-constrs}}
    \addConstraint{\secondstagevec}{\leq \secondstagerandvarcoefvec^T\randvec}{\quad [\constrdualubvec],\label{constr:var-ub}}
\end{mini!}

\noindent where $\secondstagecostvec\in\mathbb{R}^{\numsecondstagevars}$, $\secondstagerhsfixedvec\in\mathbb{R}^{\numsecondstagemainconstrs}$, $\secondstagerandvarcoefvec\in\mathbb{R}^{\numcompon}$, $\secondstagevec\in\mathbb{R}^{\numsecondstagevars}$ is a vector of second-stage decisions with domain $\recoursefeasregion$, and $\secondstagematrixfixed\in\mathbb{R}^{\numsecondstagemainconstrs\times\numsecondstagevars}$. 
The vectors $\constrdualvec\in\mathbb{R}^{\numsecondstagemainconstrs}$ and $\constrdualubvec\in\mathbb{R}^{\numcompon}$ represent vectors of dual multipliers, which will used as we describe the solution method.
Formulation \eqref{mod:recourse} is a fairly general linear program and can represent many commonly-studied operator problems in tri-level defender-attacker formulations such as the shortest path problem, the maximum flow problem, the minimum cost flow problem, and transportation problems.

\section{Structural Properties}\label{sec:struct-properties}
In this section, we illustrate several structural properties that exist in problem \eqref{mod:dad-exp-recourse-fn}. Specifically, we extend the results provided in \cite{medalaffar24} to the tri-level defender-attacker setting.

Given that $\recfn$ is convex over $\randvecsupport$ (from classic linear programming theory), for a fixed $\attackvecfix$ and $\firststagevecfix$, a \textit{mean value recourse problem} gives a lower bound on the true objective value due to Jensen's inequality, i.e., 
$
    \recfn(\bar{\randvec}(\firststagevecfix, \attackvecfix)) \leq \ev[\recfn(\randvec)|\firststagevecfix, \attackvecfix],
$
where $\randvecevfn{\firststagevecfix,\attackvecfix} = \ev[\randvec|\firststagevecfix, \attackvecfix] = (\randvarevfn{\firststagevecfix,\attackvecfix})_{\compon=1}^\numcompon$, $\randvarevfn{\firststagevecfix,\attackvecfix} = \ev[\randvark{\compon}|\firststagevecfix, \attackvecfix]$, and $\ev[\recfn(\randvec)|\firststagevecfix, \attackvecfix] = \sum_{\randvec \in \randvecsupport} \probability[\randvec|\firststagevecfix, \attackvecfix 
]\,\recfn(\randvec)$. 

Following \cite{Cormican.1998}, \cite{medalaffar24} showed that this lower bound  can be extended into a partition of the support $\randvecsupport$. Let $\partition=\{\partitioncellk,\dots,\partitioncell[\numcells]\}$ be a \textit{partition} of the support $\randvecsupport$, i.e., $\cup_{\cell=1}^\numcells \partitioncellk = \randvecsupport$ and $\partitioncell[\cellalt]\cap\partitioncellk=\emptyset$ for $\cellalt\neq\cell$. The elements of $\partition$, $\partitioncellk$, are called \textit{cells}. For a cell $\partitioncell$, the conditional expectation of $\randvec$ is $ \randvecevcell:=\condEV{\randvec}{\randvec\in \partitioncell,\firststagevec,\attackvec}$, the conditional expectation of $\recfn(\randvec)$ is $ \recourseevcell:=\condEV{\recfn(\randvec)}{\randvec\in \partitioncell,\firststagevec,\attackvec}$, and the conditional probability is $\probfncell{\firststagevec,\attackvec}:=\condprobfncell{\randvec\in \partitioncell}{\firststagevec,\attackvec}$. A \textit{refinement} of partition $\partition$ is another partition $\partition'$ such that for any cell $\partitioncell[\cell']\in\partition'$, $\partitioncell[\cell']\subseteq\partitioncellk$ for some cell $\partitioncellk\in\partition$ and $\partitioncell[\cell']\subset\partitioncellk$ for at least one cell $\partitioncellk\in\partition$ and $\partitioncell[\cell']\in\partition'$. (In the case that the support $\randvecsupport$ is a function of $(\firststagevec,\attackvec)$, then we substitute a new support $\tilde{\randvecsupport}$ such that $\tilde{\randvecsupport}\supseteq \jointsupport((\firststagevec,\attackvec))$ for all $(\firststagevec,\attackvec)\in (\firststagefeasregion,\attackfeasregion)$, where $(\firststagefeasregion,\attackfeasregion)$ is the joint feasible region.) The following result provides a lower bound on $\exprecfn(\firststagevecfix,\attackvecfix)$ with respect to a partition $\partition$, i.e., $\underline{\exprecfn}(\firststagevecfix,\attackvecfix;\partition)$ and shows that this bound is monotonic with respect to refinement.

The properties we present rely on the fact that partitioning the support of a random variable tighten's Jensen's inequality.

\setcounter{theorem}{0}

\begin{proposition}{Partitioning Tighten's Jensen's Inequality (see \cite{huang1977bounds}).}\label{lem:jensen}

Let $\probability^{\partitioncell}=\probability[\randvec\in \partitioncell]$ and $\bar{\randvec}^{\partitioncell}=\condEV{\randvec}{\randvec\in \partitioncell}$.
For partitions $\partition$ and $\partition'$ of $\randvecsupport$ such that $\partition$ is a refinement of $\partition'$, the following holds for a convex function $\recfn$:
$
\cellssum\probability^{\partitioncell}\, \recfn(\bar{\randvec}^{\partitioncell}) \geq   \sum_{\partitioncell\in\partition'} \probability^{\partitioncell}\, \recfn(\bar{\randvec}^{\partitioncell}).
$
Also, the following holds for a concave function $\recfnpenalty$:
$
\sum_{\partitioncell\in\partition'}\probability^{\partitioncell}\, \recfnpenalty(\bar{\randvec}^{\partitioncell}) \leq   \sum_{\partitioncell\in\partition} \probability^{\partitioncell}\, \recfnpenalty(\bar{\randvec}^{\partitioncell}).
$

\end{proposition}

To obtain upper bounds on the expected recourse function, we need an alternate form of the recourse function.

\begin{definition}
    A concave function $\recfnpenalty$ such that $\recfnpenalty(\randvec)=\recfn(\randvec)$ for all $\randvec\in\randvecsupport$ is called a \textit{concave equivalent recourse function}. (See \eqref{mod:recourse-penalty} for more details.)
\end{definition}

Next, we use this result to establish a monotonic bound for the expected recourse function $\exprecfn$ using a concave equivalent recourse function $\recfnpenalty$.

\begin{lemma}{Monotonic Bounds for Recourse  Objective.}\label{lem:lb}

For fixed $\firststagevecfix$ and $\attackvecfix$ and partitions $\partition$ and $\partition'$ of $\randvecsupport$ such that $\partition$ is a refinement of $\partition'$, the following holds:

\begin{align} 
\exprecfn(\firststagevecfix,\attackvecfix) =    \cellssum\probfncell{\firststagevecfix,\attackvecfix}\,\recourseevcell %\label{mod:first-equation} 
 &\geq   \cellssum\probfncell{\firststagevecfix,\attackvecfix}\, \recfn(\randvecevcellfn{\firststagevecfix,\attackvecfix}) = \underline{\exprecfn}(\firststagevecfix,\attackvecfix;\partition) \label{mod:bound-2} \\
 &\geq   \sum_{\partitioncell\in\partition'} \probfncell{\firststagevecfix,\attackvecfix}\, \recfn(\randvecevcellfn{\firststagevecfix,\attackvecfix}) = \underline{\exprecfn}(\firststagevecfix,\attackvecfix;\partition'). \label{mod:bound-3}
\end{align}

For concave equivalent recourse function $\recfnpenalty$:

\begin{equation}\label{eqn:penalty-ub}
    \exprecfn(\firststagevecfix,\attackvecfix)     
 \leq   \sum_{\partitioncell\in\partition'}\probfncell{\firststagevecfix,\attackvecfix}\, \recfnpenalty(\randvecevcellfn{\firststagevecfix,\attackvecfix}) = \bar{\exprecfn}(\firststagevecfix,\attackvecfix;\partition') \leq   \sum_{\partitioncell\in\partition} \probfncell{\firststagevecfix,\attackvecfix}\, \recfnpenalty(\randvecevcellfn{\firststagevecfix,\attackvecfix}) = \bar{\exprecfn}(\firststagevecfix,\attackvecfix;\partition).
\end{equation}

\end{lemma}

\begin{proof}
The first equality follows from the law of total expectation, and the first inequality immediately follows from Jensen's inequality. The second inequality follows from Lemma \ref{lem:jensen}. The results in \eqref{eqn:penalty-ub} are obtained similarly.
\end{proof}

The following result, without proof, shows that for a complete refinement, i.e., $\partition=\randvecsupport$, the lower bound is tight.

\begin{corollary}{Tight Recourse Bounds for Finite Support}\label{cor:tight-bounds}

For fixed $\firststagevecfix$ and $\attackvecfix$ the following holds:

\begin{equation} 
\underline{\exprecfn}(\firststagevecfix,\attackvecfix;\randvecsupport) = \sum_{\partitioncell\in\randvecsupport}\probfncell{\firststagevecfix,\attackvecfix}\, \recfn(\randvecevcellfn{\firststagevecfix,\attackvecfix}) = \exprecfn(\firststagevecfix,\attackvecfix) = \sum_{\partitioncell\in\randvecsupport}\probfncell{\firststagevecfix,\attackvecfix}\, \recfnpenalty(\randvecevcellfn{\firststagevecfix,\attackvecfix})
= \bar{\exprecfn}(\firststagevecfix,\attackvecfix;\randvecsupport)
\end{equation}

\end{corollary}

For a fixed defender solution $\firststagevecfix$, let $\attackobjvar(\firststagevecfix)^*$ denote the attacker's optimal objective. 
The following result, which provides an upper bound in the attacker's problem, immediately follows from Lemma \ref{lem:lb}.

\begin{corollary}{Monontic Upper Bound on Attacker's Optimal Objective Value.}\label{cor:attacker-ub}

For partitions $\partition$ and $\partition'$ of $\randvecsupport$ such that $\partition$ is a refinement of $\partition'$, the following holds for a concave equivalent recourse function $\recfnpenalty$:

\begin{align} 
\attackobjvar(\firststagevecfix)^* =  \max_{\attackvec\in\attackfeasregion}  \cellssum\probfncell{\firststagevecfix,\attackvec}\,\condEV{\recfnpenalty(\randvec)}{\randvec\in \partitioncell,\firststagevecfix,\attackvec} &\geq  \max_{\attackvec\in\attackfeasregion}  \cellssum\probfncell{\firststagevecfix,\attackvec}\, \recfnpenalty(\randvecevcellfn{\firststagevecfix,\attackvec}) = \attackobjvarlb(\firststagevecfix;\partition)
 \label{mod:bound-4}\\
  &\geq   \max_{\attackvec\in\attackfeasregion}\sum_{\partitioncell\in\partition'} \probfncell{\firststagevecfix,\attackvec}\, \recfnpenalty(\randvecevcellfn{\firststagevecfix,\attackvec}) = \attackobjvarlb(\firststagevecfix;\partition'). \label{mod:bound-5}
\end{align}

For finite $\randvecsupport$:
\begin{equation}
\attackobjvar(\firststagevecfix)^* =  \max_{\attackvec\in\attackfeasregion}  \cellssum\probfncell{\firststagevecfix,\attackvec}\,\condEV{\recfnpenalty(\randvec)}{\randvec\in \partitioncell,\firststagevecfix,\attackvec} =  \max_{\attackvec\in\attackfeasregion}  \sum_{\partitioncell\in\randvecsupport}\probfncell{\firststagevecfix,\attackvec}\, \recfnpenalty(\randvecevcellfn{\firststagevecfix,\attackvec}) = \attackobjvarlb(\firststagevecfix;\randvecsupport)
.
\end{equation}

\end{corollary}

%Note that for the \textit{singleton} partition (i.e.,  $\partitioncell[1]=\randvecsupport$), the upper bound problem \eqref{mod:bound-2} corresponds to the well-known \textit{mean value problem (MVP)}.

Let $\attackplansset\subseteq\attackfeasregion$ be a subset of feasible attacker solutions, and let $\objvar(\attackplansset)^*$ be the optimal objective of the original problem \eqref{mod:dad-exp-recourse-fn} when the attacker is restricted to solutions in $\attackplansset$. The next result, which follows from Lemma \ref{lem:lb} and Corollary \ref{cor:tight-bounds}, forms the basis of our solution methodology.

\begin{proposition}{Monotonic Lower Bound on Defender-Attacker Problem}\label{prop:defender-lb}

For partitions $\partition$ and $\partition'$ of $\randvecsupport$ such that $\partition$ is a refinement of $\partition'$, the following holds:

\begin{align} 
\objvar(\attackplansset)^* =  \min_{\firststagevec\in\firststagefeasregion}  \max_{\attackvec\in\attackplansset} \cellssum\probfncell{\firststagevec,\attackvec}\,\recourseevcell
 &\geq  \min_{\firststagevec\in\firststagefeasregion} \max_{\attackvec\in\attackplansset} \cellssum\probfncell{\firststagevec,\attackvec}\, \recfn(\randvecevcellfn{\firststagevec,\attackvec}) = \objvarlb(\attackplansset;\partition), \label{mod:da-bound-1}\\
 &\geq  \min_{\firststagevec\in\firststagefeasregion} \max_{\attackvec\in\attackplansset} \sum_{\partitioncell\in\partition'}\probfncell{\firststagevec,\attackvec}\, \recfn(\randvecevcellfn{\firststagevec,\attackvec}) = \objvarlb(\attackplansset;\partition'), \label{mod:da-bound-2}
\end{align}

For finite $\randvecsupport$:

\begin{equation}
    \objvar(\attackplansset)^* =  \min_{\firststagevec\in\firststagefeasregion}  \max_{\attackvec\in\attackplansset} \cellssum\probfncell{\firststagevec,\attackvec}\,\recourseevcell = \min_{\firststagevec\in\firststagefeasregion}\max_{\attackvec\in\attackfeasregion}  \sum_{\partitioncell\in\randvecsupport}\probfncell{\firststagevec,\attackvec}\, \recfn(\randvecevcellfn{\firststagevec,\attackvec}) = \objvarlb(\attackplansset;\randvecsupport).
\end{equation}

\end{proposition}

The following shows that the original problem is recovered when the partition is completely refined and all feasible attacker solutions are included.

\begin{corollary}%{Tight Lower Bound for Defender-Attacker Problem}
\label{cor:tight-lb-da}
$\objvar(\attackplansset)^* = \objvarlb(\attackplansset;\randvecsupport)$ for all $\attackplansset\subseteq\attackfeasregion$ and $\objvarlb(\attackfeasregion;\randvecsupport)^*=\objvar(\attackfeasregion)=\objvar^*$.
\end{corollary}

\section{Successive Refinement Framework}\label{sec:methods}

In this section, we present a successive refinement framework that can be applied to general defender-attacker games. First, Corollary \ref{cor:tight-lb-da} motivates an approach in which attack solutions are dynamically generated and the partition is dynamically refined. In such an approach we solve a \textit{restricted problem} for a set of attack plans $\attackvec^1,\dots,\attackvec^{\numattackplans}$ and partition $\partition$, which yields a lower bound. In this formulation, we also add variables and constraints to model the recourse objective $\recfn(\randvecevcellfn{\firststagevec,\attackvec})$.

\begin{mini!}|s|
    {\firststagevec\in\firststagefeasregion,\auxvar\geq 0,\secondstagevec}{\auxvar \label{obj:defender-epi}}{\label{mod:defender-epi-resticted}}{}
    \addConstraint{\auxvar}{\geq  \sum_{\partitioncell\in\partition}\probfncell{\firststagevec,\attackvec^\attackplan}\, (\secondstagecostvec^T\secondstagevec^{\partitioncell,\attackplan}}){\quad\forall \attackplansinlist\label{eqn:attack-cut-epi-restricted}}
    \addConstraint{\secondstagematrixfixed\secondstagevec^{\partitioncell,\attackplan}}{= \secondstagerhsfixedvec}{\quad \forall \partitioncell\in\partition, \attackplansinlist\label{constr:fixed-constrs-defender}}
    \addConstraint{\secondstagevec^{\partitioncell,\attackplan}}{\leq \secondstagerandvarcoefvec^T\,\randvecevcellfn{\firststagevec,\attackvec^\attackplan}}{\quad \forall \partitioncell\in\partition, \attackplansinlist.\label{constr:var-ub-defender}}
    \end{mini!}

Solving \eqref{mod:defender-epi-resticted} for partition $\partition$ yields a defender's solution $\firststagevecfix$. With $\firststagevecfix$ fixed, the attacker's \textit{approximate} best response can be obtained by solving the following formulation in which we use the common \textit{dualize-and-combine} approach to replace the inner minimization problem in $\recfn(\randvecevcellfn{\firststagevec,\attackvec})$ with its dual for each cell $\partitioncell$:

\begin{maxi}|s|
    {\attackvec\in\attackfeasregion,\constrdualvec,\constrdualubvec}{\sum_{\partitioncell\in\partition}\probfncell{\firststagevecfix,\attackvec}\,(\constrdualvec^{\partitioncell})^T\secondstagerhsfixedvec + \secondstagerandvarcoefvec^T\,\randvecevcellfn{\firststagevecfix,\attackvec})}{}{\label{mod:attacker}}
    \addConstraint{\secondstagematrixfixed^T\constrdualvec^{\partitioncell} + \randvecevcellfn{\firststagevecfix,\attackvec}}{\geq \secondstagecostvec}{\quad \forall \partitioncell\in\partition}
\end{maxi}

%\cite{medalaffar24} present a successive refinement algorithm for solving formulations of the form of \eqref{mod:attacker} in which refinements are made dynamically within a branch-and-cut algorithm. 
Algorithm \ref{alg:sra-attacker} (see appendix) describes an alternate version in which refinements are made after solving \eqref{mod:attacker} completely. The rationale for this approach is that when the number of elements in the partition is small, then \eqref{mod:defender-epi-resticted} and \eqref{mod:attacker} have a small number of multilinear terms, each with small degree, making the formulations much easier to solve. The partition is successively refined until lower and upper bounds are sufficiently close. In line \ref{line:compute-error} of Algorithm \ref{alg:sra-attacker}, the \textit{probability-weighted approximation error} is computed for each cell $\partitioncell\in\partition$, motivated by the fact that, for a given partition $\partition$ and fixed solution $(\firststagevecfix,\attackvecfix)$, the expected recourse objective has the following bounds:

\begin{align} 
\ubar{\exprecfn}(\firststagevecfix,\attackvecfix;\partition) = \cellssum\probfncell{\firststagevecfix,\attackvecfix}\, \recfn(\randvecevcellfn{\firststagevecfix,\attackvecfix}) \leq \exprecfn(\firststagevecfix,\attackvecfix)    %\cellssum\probfncell{\firststagevecfix,\attackvecfix}\,\condEV{\recfn(\randvec)}{\randvec\in \partitioncell,\firststagevecfix,\attackvecfix} \label{mod:first-equation-ub} \\ 
 \leq   \cellssum\probfncell{\firststagevecfix,\attackvecfix}\, UB^{\partitioncell}(\firststagevecfix,\attackvec)= \bar{\exprecfn}(\firststagevecfix,\attackvecfix;\partition), \label{mod:ub-1}
\end{align}

\noindent where $UB^{\partitioncell}(\firststagevecfix,\attackvec)$ is an upper bound on the true expected recourse value for $\partitioncell$, i.e., $\condEV{\recfn(\randvec)}{\randvec\in \partitioncell,\firststagevecfix,\attackvec}$. Section \ref{sec:compute-ub} describes different methods for computing $UB^{\partitioncell}(\firststagevecfix,\attackvec)$. Thus, an estimate of the gap between the upper and lower bounds is:
$\cellssum\probfncell{\firststagevecfix,\attackvec} (UB^{\partitioncell}(\firststagevecfix,\attackvec) - \recfn(\randvecevcellfn{\firststagevecfix,\attackvec})
$, and an estimate of the probability-weighted approximation error for a cell $\partitioncell$ is:

\begin{equation}
\textsc{Error}(\partitioncell,\firststagevecfix,\attackvec) = \probfncell{\firststagevecfix,\attackvec} (UB^{\partitioncell}(\firststagevecfix,\attackvec) - \recfn(\randvecevcellfn{\firststagevecfix,\attackvec}),
\label{eqn:approx-error-cell}
\end{equation}

\noindent If the maximum of these weighted errors exceeds a threshold $\epsilon$, then the partition is refined in line \ref{line:refine} of Algorithm \ref{alg:sra-attacker}, as described in \S\ref{sec:refine-cell}. Algorithm \ref{alg:sra-da} (see appendix) describes how to use the successive refinement algorithm for the defender-attacker problem. This algorithm alternates between solving the main problem \eqref{mod:defender-epi-resticted} and solving the attacker's subproblem \eqref{mod:attacker}. For each new attacker's solution generated, a cut is added to \eqref{mod:defender-epi-resticted}, and the partition is (possibly) refined.

\subsection{Computing Upper Bound Estimates for a Cell}\label{sec:compute-ub}

A general approach for computing the upper bound term $UB^{\partitioncell}(\firststagevecfix,\attackvec)$ in \eqref{eqn:approx-error-cell} is via \textit{Monte Carlo sampling}, i.e.,
$UB^{\partitioncell}(\firststagevecfix,\attackvec) = \frac{1}{N'}\sum_{i=1}^{N'} \recfn(\randvec^i),
$
where $(\randvec^i)_{i=1}^{N'}$ is a sample of size $N'$ from the distribution defined by $\probfncell{\firststagevecfix,\attackvec}=\condprobfncell{\randvec\in \partitioncell}{\firststagevecfix,\attackvec}$. However, for problems with special \textit{variable upper bound} structure, an upper bound can be computed without sampling. Specifically, this special structure exists in problems in which upper bounds on variables in the recourse function depend on $\randvec$, as in \eqref{mod:recourse}. This structure exists in many network flow recourse problems in which the arcs fail, and the variables correspond to arcs. This is the case for shortest path, maximum flow, and minimum flow recourse problems. When this special structure exists, the following \textit{penalty recourse problem can be formulated} (see \cite{Cormican.1998} and \cite{Smith.2020}):

\begin{mini!}|s|
    {\secondstagevec\in \recoursefeasregion}{\secondstagecostvec^T\secondstagevec + 
    \secondstagevec M \randvec
    \label{eq:obj-recourse-penalty}}{\label{mod:recourse-penalty}}{\recfnpenalty(\randvec) =}
    \addConstraint{\eqref{constr:fixed-constrs},}
\end{mini!}

\noindent where $M = \text{diag}(\bar{\constrdualubvec})$, where $\bar{\constrdualubvec}$ is a vector of upper bounds on the dual variables of \eqref{constr:var-ub}. \cite{Cormican.1998} showed that for an interdiction problem with maximum flow recourse, $\recfnpenalty(\randvec)=\recfn(\randvec)$ for all $\randvec\in\randvecsupport$. Also, by linear programming theory, $\recfnpenalty$ is concave.
So, for problems with variable upper bound structure, Lemma \ref{lem:lb} implies that an upper bound for a cell can be computed as:

\begin{equation}
    UB^{\partitioncell}(\firststagevecfix,\attackvec) = \recfnpenalty(\randvecevcellfn{\firststagevecfix,\attackvec}).
\end{equation}
%Thus, using Jensen's inequality, the following result provides an upper bound on the expected recourse objective.

% \begin{corollary}{Monotonic Upper Bound for Recourse  Objective.}\label{cor:lb}

% For fixed solution $(\firststagevecfix,\attackvecfix)$, partitions $\partition$ and $\partition'$ of $\randvecsupport$ such that $\partition$ is a refinement of $\partition'$, the following holds:

% \begin{align} 
% \exprecfn(\firststagevecfix,\attackvecfix) &=    \cellssum\probfncell{\firststagevecfix,\attackvecfix}\,\condEV{\recfn(\randvec)}{\randvec\in \partitioncell,\firststagevecfix,\attackvecfix} \label{mod:first-equation-ub-2} \\ 
%  &\leq   \cellssum\probfncell{\firststagevecfix,\attackvecfix}\, \recfnpenalty(\randvecevcellfn{\firststagevecfix,\attackvecfix}) = \bar{\exprecfn}(\firststagevecfix,\attackvecfix;\partition) \label{mod:ub-2} \\
%  &\leq   \sum_{\partitioncell\in\partition'} \probfncell{\firststagevecfix,\attackvecfix}\, \recfnpenalty(\randvecevcellfn{\firststagevecfix,\attackvecfix}) = \bar{\exprecfn}(\firststagevecfix,\attackvecfix;\partition'). \label{mod:ub-3}
% \end{align}

%\end{corollary}

\subsection{Choosing How to Refine a Cell}\label{sec:refine-cell}

After choosing the cell with the maximum probability-weighted error, $\partitioncellvar^{max}$, we subdivide $\partitioncellvar^{max}$ to refine the partition. We use a rectangular partitioning approach in which we subdivide cells along the axis of one random variable in $\randvec$ (cf. \cite{Cormican.1998}). Let $\randvecsupport_\compon^{\partitioncell}$ be the support of $\randvark{\compon}$ in cell $\partitioncell$, and let $\textsc{PartitionSupport}(\partitioncell,\compon)$ be a function that partitions $\randvecsupport_\compon^{\partitioncell}$, replacing $\partitioncell$ with a set of new sub-cells. In the case that $\randvec$ is a vector of Bernoulli random variables (see test problem in \S\ref{sec:results}) $\randvecsupport_\compon^{\partitioncell}=\{0,1\}$ and therefore can only be partitioned once so that $\textsc{PartitionSupport}(\partitioncell,\compon')$ will return two new disjoint subcells $\partitioncellvar'=\partitioncell\cap\{\randvec\,|\,\randvark{\compon'}=0\}$ and $\partitioncellvar''=\partitioncell\cap\{\randvec\,|\,\randvark{\compon'}=1\}$.
%such that $\randvark{\compon} = \randvark{\compon}' = \randvark{\compon}'' \, \forall \compon \in \components \setminus \{\compon'\}$, $\randvark{\compon'}' = 0$ and $\randvark{\compon'} '' = 1$. 
%Where $\randvark{\compon}'$ and  $\randvark{\compon}''$ represent the values of the random variable $\randvark{\compon}$ in the new subcells $\partitioncellvar'$ and $\partitioncellvar''$. 
For the case in which $\randvec$ is a vector of random variables with a support of more than two elements, there are several different options of partitioning $\randvecsupport_\compon^{\partitioncell}$ (e.g., fully partitioning $\randvecsupport_\compon^{\partitioncell}$ to produce a new cell for each element in $\randvecsupport_\compon$ or partially partitioning $\randvecsupport_\compon^{\partitioncell}$ to result in cells consisting of subsets of $\randvecsupport_\compon$). For a fixed solution $(\firststagevecfix,\attackvecfix)$, we choose the random variable $\randvark{\compon}$ to subdivide on as the one that minimizes the resulting weighted error: $\compon'\in \arg\min_{\compon=1,\dots,\numcompon}\,\textsc{Error}(\textsc{PartitionSupport}(\partitioncellvar^{max},\compon),\firststagevecfix,\attackvecfix)$.

The function $\textsc{Refine}(\partition,\firststagevecfix,\attackvecfix,\partitioncellvar^{max})$ produces a new refinement $\partition'$ with cells that does not include $\partitioncellvar^{max}$ but includes its subcells, i.e., $\textsc{Refine}(\partition,\firststagevecfix,\attackvecfix,\partitioncellvar^{max}) = (\partition \setminus {\partitioncellvar^{max}}) \cup \textsc{PartitionSupport}(\partitioncellvar^{max},\compon')$. After the partition is refined, formulation \eqref{mod:defender-epi-resticted} is modified by replacing $\partition$ with $\partition'$ in (\ref{constr:fixed-constrs-defender}) and (\ref{constr:var-ub-defender}), and formulation \eqref{mod:attacker} is modified by replacing the partition in the objective function summation.

\section{Application of Framework to Stochastic Interdiction Problems with Decision-Dependent Failure Probabilities}\label{sec:application}

In this section, we demonstrate how the SRA framework described in the previous section can be applied to a class of stochastic interdiction problems in which the probability distribution of an component's capacity depends on both the defender's and attacker's allocation of resources to that component. The key assumption underlying the approach described in this section is that the random variables in $\randvec$ are \textit{discrete, mutually statistically independent, and have a finite support}.

First, we show how the SRA framework applies to problems in which the random variables in $\randvec$ are discrete and statistically independent so the probability of a realization of $\randvec$ is the product of the realization of the random variables in $\randvec$, i.e.,
$
	\condprobfn{\randvec}{\firststageallocvec,\attackvec} = \probprod \condprobfn{\randvark{\compon}}{\firststageallocvec,\attackvec}.
$
This product form illustrates why solving stochastic programs with decision-dependent probabilities is challenging. Even when the $\condprobfn{\randvark{\compon}}{\firststageallocvec,\attackvec}$ terms are linear in $(\firststageallocvec,\attackvec)$, this expression still consists of one or more (non-convex) multilinear expressions with degree up to $\numcompon$. In the SRA method, these terms are replaced with $\probfncell{\firststagevec,\attackvec}$. When the partition $\partition$ is small, then the $\probfncell{\firststagevec,\attackvec}$ terms have a small number of these multilinear expressions, each with small degree.

For this class of problems, the refinements of the partition $\partition$ can be made dynamically during a branch-and-bound (or spatial branch-and-bound) algorithm. To facilitate dynamic partition refinements, it is necessary to reformulate both the defender's problem \eqref{mod:defender-epi-resticted} and the attacker's problem \eqref{mod:attacker}. 

To reformulate the attacker's problem \eqref{mod:attacker}, it is useful to think of a partition $\partition$ as a tree, rooted at a node $0$, in which new branches are formed when a cell is refined on a single random variable from $\randvec$. Thus, each cell in a partition corresponds to a leaf node in a tree, and the other nodes represent cells from previous partitions. Let $\auxvar_\treenode$ be a variable that models the quantity $\probfncellindex{\firststagevecfix,\attackvec}\,\recfn(\randvecevcellfnindex{\firststagevecfix,\attackvec})$. 
Let $\parttreenodes(\partition)$ be set of nodes in the tree, and $\parttreeleafnodes(\partition)$ the set of leaf nodes, i.e., cells in $\partition$. (Hereafter, we use the index $\leaf$ to refer to a partition cell $\partitioncellk$.) We assume that the random variables in $\randvec$ are independent, so that $\probfncell{\firststagevec,\attackvec} =\prod_{\compon=1}^\numcompon\condprobfn{\randvark{\compon}\in\randvecsupport^{\partitioncell}_\compon}{\firststagevec,\attackvec}$, where $\randvecsupport^{\partitioncell}_\compon$ is the support of random variable $\randvark{\compon}$ in cell $\partitioncell$. 
Define the conditional probabilities $\condtreeprob{\firststagevec,\attackvec}=\condprobfn{\randvec\in \partitioncell[\treenode']}{\randvec\in \partitioncellk,\firststagevec,\attackvec}$, and let $\descendentnodes(\treenode;\partition)$ be the set of descendants of node $\treenode$ in the tree defined by $\partition$. Then, for a fixed defender solution $\firststagevecfix$, we formulate the attacker's problem by replacing the recourse function $\recfn$ with the dual of the concave equivalent version for cell $\cell$:

\begin{maxi!}|s|
    {\attackvec\in\attackfeasregion}{\auxvar_0 \label{eq:epi-obj}}{\label{mod:epi-attacker}}{}
    \addConstraint{\auxvar_\treenode}{\leq \sum_{\treenode'\in \descendentnodes(\treenode;\partition)} \condtreeprob{\firststagevecfix,\attackvec}\,\auxvar_{\treenode'}}{\quad \forall \treenode\in \parttreenodes(\partition)\setminus \parttreeleafnodes(\partition)\label{eqn:conditional-epigraph-attacker}}
    \addConstraint{\auxvar_\leaf}{\leq \secondstagerhsfixedvec^T\constrdualvec^{\cell}}{\quad\forall \leaf\in \parttreeleafnodes(\partition)\label{eqn:leaf-cut-attacker}}
    \addConstraint{\secondstagematrixfixed^T\constrdualvec^{\cell}}{\geq \secondstagecostvec + M \randvecevcellfnindex{\firststagevecfix,\attackvec}}{\quad \forall \cell=1,\dots,\numcells.\label{eqn:leaf-constr-attacker}}
\end{maxi!}

Relying on the assumption that the random variables in $\randvec$ are independent, constraints \eqref{eqn:conditional-epigraph-attacker} recursively compute the probability of leaf/cell $\cell$ times the expectation of the recourse function conditioned on cell $\cell$. 
Constraints \eqref{eqn:leaf-cut-attacker} and \eqref{eqn:leaf-constr-attacker} compute the recourse objective for node $\treenode$. 
Note that refining $\partition$ to obtain $\partition'$ means replacing some leaf node $\leaf$ with new set of leaf nodes $\descendentnodes(\treenode;\partition')$. Thus, for each $\leaf'\in \descendentnodes(\treenode;\partition')$ new variables $(\auxvar_{\leaf'},\constrdualvec^{\leaf'})$ are added along with constraints \eqref{eqn:conditional-epigraph-attacker}-\eqref{eqn:leaf-constr-attacker}. Because $\recfnpenalty$ is concave, adding these new constraints yields a new upper bound that is no larger than the previous one (see Corollary \ref{cor:attacker-ub}). Formulation \eqref{mod:epi-attacker} can be solved via a branch-and-bound algorithm in which the partition is dynamically refined upon the discovery of a new incumbent solution. Algorithm \ref{alg:sra-callback-attacker} contains pseudocode that can be executed as a callback routine when a new incumbent is found.
                    
\begin{algorithm}
	{\sc SuccessiveRefinementCallback}()
	\begin{algorithmic}[1]
        \STATE Let $\leaf'\in \arg\max_{\leaf\in \parttreeleafnodes(\partition)} \{\textsc{Error}(\leaf,\firststagevecfix,\attackvec)\}$ and $e^{max} =  \max_{\leaf\in \parttreeleafnodes(\partition)} \{\textsc{Error}(\leaf,\firststagevecfix,\attackvec)\}$,
        \IF{$e^{max} > \epsilon$}
            \STATE $\partition\gets \textsc{Refine}(\partition, \firststagevecfix, \attackvec, \leaf')$.
            \STATE Update formulation \eqref{mod:epi-attacker} based on new partition $\partition$.
        \ENDIF
	\end{algorithmic}
	\caption{Successive refinement callback routine for attacker's problem.}
	\label{alg:sra-callback-attacker}
\end{algorithm}

In a manner similar to the attacker's problem \eqref{mod:epi-attacker}, we next reformulate the defender's problem \eqref{mod:defender-epi-resticted}. Here we model the recourse problem at each leaf $\leaf$ with Bender's cuts, which eliminates the need to add variables each time the partition is refined. This alternate approach, which we use for our computational experiments, is the following formulation, where $(\hat{\constrdualvec}^{\cell\attackplan},\hat{\constrdualubvec}^{\cell\attackplan})$ denotes an optimal dual solution to the recourse problem $\recfn(\randvecevcellfn{\firststagevecfix,\attackvec^\attackplan})$ for a fixed defender solution $\firststagevecfix$ and attacker solution $\attackvec^\attackplan$.

\begin{mini!}|s|
    {\firststagevec\in\firststagefeasregion}{\auxvar_0 \label{eq:epi-obj-defender-benders}}{\label{mod:epi-defender-benders}}{}
    \addConstraint{\auxvar_0}{\geq \auxvar_{0,\attackplan}}{\quad\attackplansinlist\label{eqn:conditional-epigraph-root-benders}}
    \addConstraint{\auxvar_{\treenode,\attackplan}}{\geq \sum_{\treenode'\in \descendentnodes(\treenode;\partition)} \condtreeprob{\firststagevec,\attackvec^\attackplan}\,\auxvar_{\treenode',\attackplan}}{\quad \forall \treenode\in \parttreenodes(\partition)\setminus \parttreeleafnodes(\partition),\attackplansinlist\label{eqn:conditional-epigraph-benders}}
    \addConstraint{\auxvar_{\leaf,\attackplan}}{\geq \secondstagerhsfixedvec^T\hat{\constrdualvec}^{\cell\attackplan} + ((\hat{\constrdualubvec}_\compon^{\cell\attackplan}\secondstagerandvarcoef_\compon)_{\compon=1}^\numcompon)^T\,\randvecevcellfnindex{\firststagevec,\attackvec^\attackplan}}{\quad\forall \leaf\in \parttreeleafnodes(\partition),\attackplansinlist\label{eqn:leaf-cut-benders}}
\end{mini!}

As in \eqref{mod:epi-attacker}, constraints \eqref{eqn:conditional-epigraph-attacker} recursively compute the probability of leaf/cell $\cell$ times the expectation of the recourse function conditioned on cell $\cell$ for each attacker solution $\attackvec^1,\dots,\attackvec^{\numattackplans}$. 
Constraints \eqref{eqn:leaf-cut-benders} compute the recourse objective for node $\treenode$ and attacker solution $\attackvec^\attackplan$ using Bender's cuts. 
As in \eqref{mod:epi-attacker}, refining $\partition$ to obtain $\partition'$ means replacing some leaf node $\leaf$ with new set of leaf nodes $\descendentnodes(\treenode;\partition')$. Thus, for each $\leaf'\in \descendentnodes(\treenode;\partition')$ new variables $(\auxvar_{\leaf',\attackplan},\constrdualvec^{\leaf',\attackplan},\constrdualubvec^{\leaf',\attackplan})$ are added along with constraints \eqref{eqn:conditional-epigraph-benders}-\eqref{eqn:leaf-cut-benders}. Because $\recfn$ is concave, adding these new constraints yields a new lower bound that is no smaller than the previous one (see Proposition \ref{prop:defender-lb}). Algorithm \ref{alg:sra-callback} contains pseudocode for solving \eqref{mod:epi-defender-benders} using the successive refinement algorithm. In the next subsection, we describe how to apply Algorithms \ref{alg:sra-callback-attacker} and \ref{alg:sra-callback} to problems with different types of defender and attacker allocation variables.

\begin{algorithm}
	{\sc SuccessiveRefinementAlgorithmCallback}()
	\begin{algorithmic}[1]
        \STATE For each $\leaf\in \parttreeleafnodes(\partition)$, solve \eqref{mod:epi-attacker} to obtain an attack solution $\attackvec^*$
        \STATE solve (\ref{mod:recourse}) using the incumbent $\firststagevec$ and $\attackvec^*$ to obtain dual multipliers $(\constrdualvec^\leaf, \constrdualubvec^\leaf)$ and add cuts to (\ref{mod:epi-defender-benders})
        \STATE Let $\leaf'\in \arg\max_{\leaf\in \parttreeleafnodes(\partition)} \{\textsc{Error}(\leaf,\firststagevec,\attackvecfix)\}$ and $e^{max} =  \max_{\leaf\in \parttreeleafnodes(\partition)} \{\textsc{Error}(\leaf,\firststagevec,\attackvecfix)\}$,
        \IF{$e^{max} > \epsilon$}
            \STATE $\partition \gets \textsc{Refine}(\partition,\firststagevec, \attackvecfix,\leaf')$. 
            \STATE Update formulation \eqref{mod:epi-defender-benders} based on new partition $\partition$.
        \ENDIF
	\end{algorithmic}
	\caption{Successive refinement algorithm callback routine for defender-attacker problem.}
	\label{alg:sra-callback}
\end{algorithm}

\subsection{Application to Different Types of Allocation Variables}\label{sec:app-allocation-var-types}
In this section, we describe how the SRA method applies to stochastic interdiction problems with different types of allocation variables.

\subsubsection{Binary Allocation Variables}\label{sec:app-binary}

Here, we assume that both the defender and attacker allocate a discrete amount of units. For ease of exposition, but without loss of generality, we assume that each random variable $\randvark{\compon}$ corresponds to the capacity of a component $\compon$ in the system (e.g., an arc in a network) and that the probability distribution of $\randvark{\compon}$ depends only on the amount of defense and attack units allocated to component $\compon$. Under this assumption, the defender's allocation can be represented as the sum of binary variables multiplied by coefficients, i.e., $\alloclevelssum \alloclevel \firststageallocil$, where $\firststageallocil$ is a binary variable that is $1$ if $\alloclevel$ units are allocated to component $\compon$ and $0$ otherwise.  (Similarly for $\attackvec$.) Given binary $\firststagevec$ and $\attackvec$, the probability of a realization can be represented as: $
	\condprobfn{\randvec=\tilde{\randvec}}{\firststageallocvec,\attackvec} =\probprod \alloclevelssum\attackalloclevelssum \stateprobfndiscri \firststageallocil\attackallocil$
where $\stateprobfndiscri$ is the probability that $\randvark{\compon}=\tilde{\randsym}_\compon$, given that the defense allocation level is $\alloclevel$ and the attack allocation level is $\attackalloclevel$. Because the variables in $\firststageallocvec$ and $\attackvec$ are binary, the  McCormick envelopes used to linearize products of variables in $\condprobfn{\randvec=\tilde{\randvec}}{\firststageallocvec,\attackvec}$ are exact for binary solutions. 
%This assumption allows us to implement Algorithms \ref{alg:sra-attacker} and \ref{alg:sra-da} using modern branch-and-cut solvers such as Gurobi. %When the variables $\firststagevec$ and $\attackvec$ are binary, line \ref{alg-line:relaxation-attacker} in Alg. \ref{alg:sra-attacker} 
%(and \ref{alg-line:relaxation-defender} in Alg. \ref{alg:sra-da}) 
%represents linear programming relaxations and line \ref{alg-line:feasible-to-orig-attacker} in Alg. \ref{alg:sra-attacker} 
%(and \ref{alg-line:feasible-to-orig-defender} in Alg. \ref{alg:sra-da}) 
%represents integer feasibility.

\paragraph{Benchmark Method: Deterministic Equivalent Formulation}

When variables $\firststagevec$ and $\attackvec$ are binary, the original problem \eqref{mod:dad-exp-recourse-fn} can be formulated as the following deterministic equivalent formulation with a multilinear objective. First, we enumerate the realizations of $\randvec$ as $\randvec^1,\dots,\randvec^\numscens$ and make copies of $\secondstagevec$ for each scenario $\scen$. Then, we can formulate the problem using an epigraph formulation:

\begin{mini!}|s|
    {\firststagevec\in \firststagefeasregion}{\auxvar \label{eq:obj-multi}}{\label{mod:multi-program}}{}
    \addConstraint{\auxvar}{\geq \scenssum \left( \probprod \alloclevelssum\attackalloclevelssum \stateprobfndiscris \attackallocfixedil^{\attackplan}\firststageallocil\right) \secondstagecostvec^T\secondstagevec^{\scen, \attackplan}}{\quad\forall\attackplansinlist \label{constr:epi-lb}}
    \addConstraint{\secondstagematrixfixed\secondstagevec^{\scen, \attackplan}}{= \secondstagerhsfixedvec}{\quad \forall \scen = 1, \dots, \numscens \label{constr:fixed-constrs-att}}
    \addConstraint{\secondstagevec^{\scen,\attackplan}}{\leq \secondstagerandvarcoefvec^T\randvec^\scen}{\quad  \forall \scen = 1, \dots, \numscens,\label{constr:var-ub-att}}
\end{mini!}

Because it is impractical to enumerate all feasible attack solutions $\attackvecfix$ in the set $\attackfeasregion$, formulation \eqref{mod:multi-program} can be solved by dynamically adding new cuts of the form \eqref{constr:epi-lb} for selected attack solutions, $\attackvec^1,\dots,\attackvec^{\numattackplans}$ within a Benders decomposition algorithm. For a given incumbent solution $\firststageallocvecfix$, a cut is obtained from a solution to the following attacker's problem:

\begin{maxi!}|s|
    {\attackvec\in \attackfeasregion}{\scenssum \left( \probprod \alloclevelssum\attackalloclevelssum \stateprobfndiscris \firststageallocfixedil\attackallocil\right) \secondstagecostvec^T\secondstagevec^\scen \label{eq:obj-multi-attack}}{\label{mod:multi-program-attack}}{}
    \addConstraint{\secondstagematrixfixed\secondstagevec^\scen}{= \secondstagerhsfixedvec}{\quad \forall \scen = 1, \dots, \numscens \label{constr:fixed-constrs-de}}
    \addConstraint{\secondstagevec^\scen}{\leq \secondstagerandvarcoefvec^T\randvec^\scen}{\quad  \forall \scen = 1, \dots, \numscens,\label{constr:var-ub-de}}
\end{maxi!}

The variable products in \eqref{constr:epi-lb} and \eqref{eq:obj-multi-attack} can be linearized using McCormick constraints (\cite{mccormick1976computability}), which form the convex envelope. However, solving formulations \eqref{mod:multi-program} and \eqref{mod:multi-program-attack} is challenging because (i) the number of scenarios $\numscens$ increases exponentially with the number of components $\numcompon$ and (ii) there are many of multilinear variable products in \eqref{constr:epi-lb} and \eqref{eq:obj-multi-attack}, each having a degree of up to $\numcompon$.

\subsubsection{Integer or Continuous Allocation Variables}\label{subsubsec:app-integer-continuous}

When $\firststagevec$ and $\attackvec$ are integer or continuous, then the probability of a realization can be represented as: $
	\condprobfn{\randvec=\tilde{\randvec}}{\firststageallocvec,\attackvec} =\probprod \stateprobfn{\compon}{\tilde{\randsym}_\compon;\firststagevar_\compon,\attackvar_\compon},
$
where $\stateprobfn{\compon}{\tilde{\randsym}_\compon;\firststagevar_\compon,\attackvar_\compon}$ is the probability that $\randvark{\compon}=\tilde{\randsym}_\compon$, given that $\firststagevar_\compon$ defense units and $\attackvar_\compon$ attack units are allocated to component $\compon$. For example, $\randvark{\compon}$ could be a binomial random variable with success probability $\firststagevar_\compon/(\firststagevar_\compon+\attackvar_\compon+\epsilon)$, where $\epsilon$ is a small number to avoid dividing by zero. 

In most cases, $\stateprobfn{\compon}{\tilde{\randsym}_\compon;\firststagevar_\compon,\attackvar_\compon}$ is nonlinear in $\firststagevar_\compon$ and $\attackvar_\compon$. However, the product in $\condprobfn{\randvec=\tilde{\randvec}}{\firststageallocvec,\attackvec}$ can be linearized using modern solvers such as Gurobi. Therefore, Algorithms \ref{alg:sra-callback-attacker} and \ref{alg:sra-callback} can be utilized as callback routines using these modern solvers. 
%When the variables $\firststagevec$ and $\attackvec$ are pure integer, line \ref{alg-line:relaxation-attacker} in Alg. \ref{alg:sra-attacker} (and \ref{alg-line:relaxation-defender} in Alg. \ref{alg:sra-da}) represents linear programming relaxations and line \ref{alg-line:feasible-to-orig-attacker} in Alg. \ref{alg:sra-attacker} (and \ref{alg-line:feasible-to-orig-defender} in Alg. \ref{alg:sra-da}) represent integer feasibility. 
When the variables $\firststagevec$ and $\attackvec$ are continuous, 
%line \ref{alg-line:relaxation-attacker} in Alg. \ref{alg:sra-attacker} (and \ref{alg-line:relaxation-defender} in Alg. \ref{alg:sra-da}) represents convex relaxations 
%and line \ref{alg-line:feasible-to-orig-attacker} in Alg. \ref{alg:sra-attacker} (and \ref{alg-line:feasible-to-orig-defender} in Alg. \ref{alg:sra-da}) represent the 
the callback routine is executed when the solution obtained by the relaxed problem %(e.g., obtained using convex envelopes) 
is feasible to the original one.

\section{Numerical Example}\label{sec:example}

We illustrate using the SRA method to solve a defender-attacker tri-level maximum flow problem on a $3 \times 3$ network shown in Figure \ref{fig:solution-attack-4-3-0}. (\S\ref{sec:results} provides more details about this problem.) In this example, the defender and attacker have two allocation levels; thus, they can allocate 0 or 1 unit to an arc. Both the defender and attacker have a budget of four units. We assume that each component (arc) can be in either of two states and we used the following component state probabilities. 

\begin{equation}\label{state-prob-fn}
    f_k(\componoutcome;\alloclevel,\attackalloclevel) = \begin{cases}
\frac{\alloclevel}{\alloclevel + \attackalloclevel}, & \componoutcome = 1,\alloclevel + \attackalloclevel > 0\\
1- \frac{\alloclevel}{\alloclevel + \attackalloclevel}, & \componoutcome = 0,\alloclevel + \attackalloclevel> 0\\
1, & \componoutcome = 1,\attackalloclevel = 0\\
0, & \componoutcome = 0,\attackalloclevel = 0\\
0, & \componoutcome = 1,\alloclevel = 0, \attackalloclevel > 0\\
1, & \componoutcome = 0,\alloclevel = 0,\attackalloclevel > 0
\end{cases}
\end{equation}

The first two cases imply that the probability of arc failure is computed using a contest success function (\cite{hausken2011defending}). The third and fourth cases imply that an arc that is not attacked never fails. The final two cases imply that an undefended arc always fails if attacked. 

Starting with an unrefined partition $\partition=\{\randvecsupport\}$ and after the presolve phase, the solver finds an incumbent defender solution in which one unit is allocated to arcs $(1,2), (2,3), (8,9)$ (see Figure \ref{fig:solution-attack-4-3-0}). With this defender solution fixed, Algorithm \ref{alg:sra-callback-attacker} is run, identifying an incumbent solution $(1,2), (2,3), (2,5), (7,8)$ (see Figure \ref{fig:solution-attack-4-3-0}). The attacker's partition has not yet been refined, and the lower bound is currently $1.0$. Because the partition has not been refined, both the defender and attacker problems are linear and easy to solve.

Having found an incumbent solution, the next step in Algorithm \ref{alg:sra-callback-attacker} is to check if the current partition should be refined. Using \eqref{eqn:approx-error-cell}, the probability-weighted approximation gap for the only cell in the partition is computed as $2.5$, indicating that the cell should be refined. The cell is refined on arc \((1,2)\) (see Figure \ref{fig:tree-attacker-4-3-1}), which is a contested arc, i.e., an arc in which both the defender and attacker allocate resources. Because of the form of the state probability function \eqref{state-prob-fn}, only contested arcs are branched on during refinement. After this refinement, Algorithm \ref{alg:sra-callback-attacker} continues until it proves that solution \((1,2), (2,3), (2,5), (7,8)\) is optimal, with a final expected recourse objective value of $4.5$. The next step is to add cuts generated using this solution to the defender's main problem \eqref{mod:epi-defender-benders}.

With the cuts added to the main problem, Algorithm \ref{alg:sra-callback} then checks if the defender's partition needs refinement. The probability-weighted approximation gap for the only cell in the partition is $2.5$, so the cell is refined on arc \((1,2)\) (see Figure \ref{fig:tree-attacker-4-3-1}). Note that this refinement is identical to the one for the attacker's problem. After this refinement, Algorithm \ref{alg:sra-callback-attacker} is run, finding an attacker solution with an expected recourse value of $4.5$. After adding cuts based on this attacker solution, the defender's upper bound is $21.52$ and the lower bound is $4.5$, a gap of about 378\%.

After identifying a few more incumbent solutions, the defender refines its partition further, choosing to refine cell 2 along contested arc \((8,9)\) (see Figure \ref{tree-defender-6-2}). Because the attacker's problem is more refined, it has a small number of multilinear terms. This refinement reduces the defender's upper bound to $11.34$, resulting in an optimality gap of about 152\%. 

Figure \ref{fig:final-iteration} shows the final iteration of Algorithm \ref{alg:sra-callback}, starting with the new defender's incumbent solution of \((1,2), (2,3), (7,8), (8,9)\) (see Figure \ref{fig:solution-attack-9-2-0}). Currently, the upper bound is $11.16$, and the lower bound is $6.5$; so the gap has decreased to about 71.8\%. With the defender solution fixed, Algorithm \ref{alg:sra-callback-attacker} is run, finding an incumbent solution \((2,3), (5,6), (7,8), (8,9)\) (see Figure \ref{fig:solution-attack-9-2-0}). The attacker's partition has not yet been refined, and the lower bound is $3.5$. To improve the bound, the attacker refines the partition tree along contested arc \((7,8)\) (see Figure \ref{fig:tree-attacker-9-2-1}). After this refinement, the attacker finds a new incumbent solution \((1,2), (2,3), (7,8), (8,9)\) (see Figure \ref{fig:solution-attack-9-3-1}), which improves the lower bound to $4.65$. The next step is to refine cell 1 in the tree (which has an error of $2.43$) along arc $(1,2)$ (see Figure \ref{fig:tree-attacker-9-3-2}). Algorithm \ref{alg:sra-callback} terminates with defender solution \((1,2), (2,3), (7,8), (8,9)\) and attacker solution \((1,2), (2,3), (5,6), (7,8)\), resulting in an expected max flow of $10.5$. Note that the algorithm terminated after only two refinements to the defender's partition, meaning that the problem only has a small number of multiliear terms. This contrasts with the benchmark approach, which includes all multilinear terms from the outset. Starting with a less nonlinear version of the problem and incrementally introducing nonlinearity through partition refinement proves to be a more manageable strategy, which we will test in the next section.

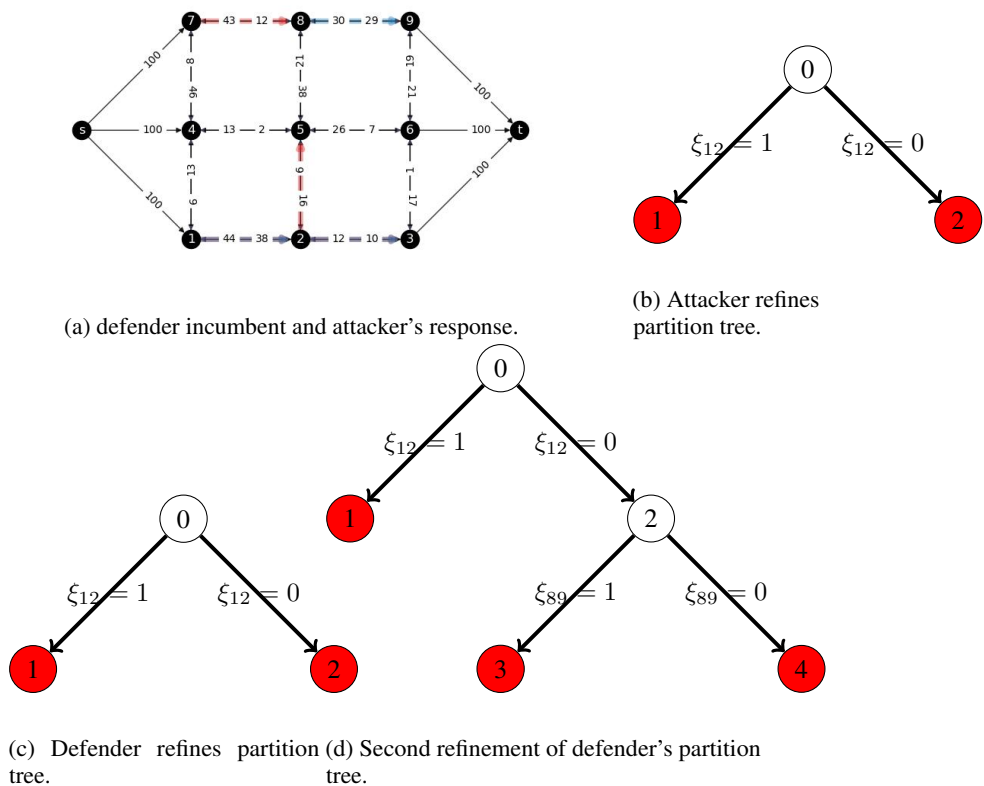
\begin{figure}
    \centering
    \begin{subfigure}[b]{0.55\textwidth}
        \centering
        \includegraphics[width=\textwidth]{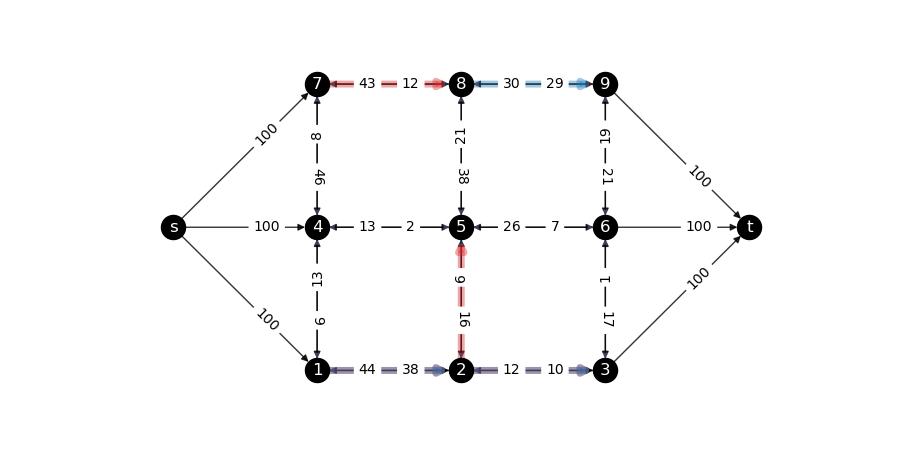}
        \caption{\small defender incumbent and attacker's response.}
        \label{fig:solution-attack-4-3-0}
    \end{subfigure}%
    \begin{subfigure}[b]{0.15\textwidth}
        \centering
        \begin{tikzpicture}
  \draw
    (0, 0) node[draw=black,circle,fill=white] (0){0}
    (-2, -2) node[draw=black,circle,fill=red] (1){1}
    (2, -2) node[draw=black,circle,fill=red] (2){2};
  \begin{scope}[->]
    \draw[line width=1.5] (0) to node[] {$\xi_{12} = 1$} (1);
    \draw[line width=1.5] (0) to node[] {$\xi_{12} = 0$} (2);
  \end{scope}
\end{tikzpicture}
        \caption{Attacker refines partition tree.}
        \label{fig:tree-attacker-4-3-1}
    \end{subfigure}
    \\
    \begin{subfigure}[b]{0.25\textwidth}
        \centering
        \begin{tikzpicture}
  \draw
    (0, 0) node[draw=black,circle,fill=white] (0){0}
    (-2, -2) node[draw=black,circle,fill=red] (1){1}
    (2, -2) node[draw=black,circle,fill=red] (2){2};
  \begin{scope}[->]
    \draw[line width=1.5] (0) to node[] {$\xi_{12} = 1$} (1);
    \draw[line width=1.5] (0) to node[] {$\xi_{12} = 0$} (2);
  \end{scope}
\end{tikzpicture}
        \caption{Defender refines partition tree.}
        \label{tree-defender-4-1}
    \end{subfigure}
    \begin{subfigure}[b]{0.35\textwidth}
        \centering
        \begin{tikzpicture}
      \draw
        (0, 0) node[draw=black,circle,fill=white] (0){0}
        (-2, -2) node[draw=black,circle,fill=red] (1){1}
        (2, -2) node[draw=black,circle,fill=white] (2){2}
        (0, -4) node[draw=black,circle,fill=red] (3){3}
        (4, -4) node[draw=black,circle,fill=red] (4){4};
      \begin{scope}[->]
        \draw[line width=1.5] (0) to node[] {$\xi_{12} = 1$} (1);
        \draw[line width=1.5] (0) to node[] {$\xi_{12} = 0$} (2);
        \draw[line width=1.5] (2) to node[] {$\xi_{89} = 1$} (3);
        \draw[line width=1.5] (2) to node[] {$\xi_{89} = 0$} (4);
      \end{scope}
    \end{tikzpicture}
        \caption{Second refinement of defender's partition tree.}
        \label{tree-defender-6-2}
    \end{subfigure}
    \caption{Iteration of SRA algorithm.}
\end{figure}

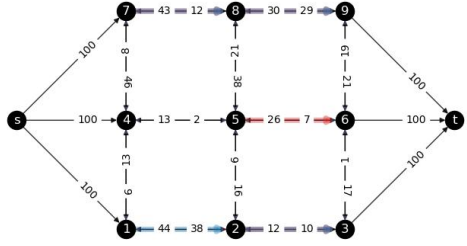
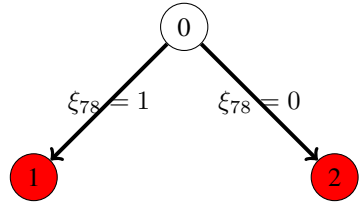
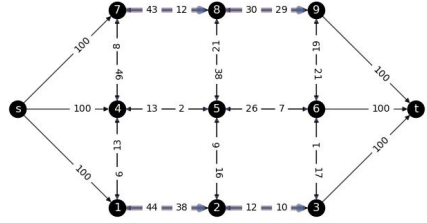
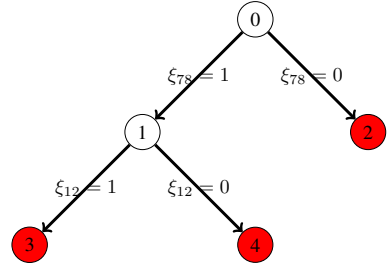
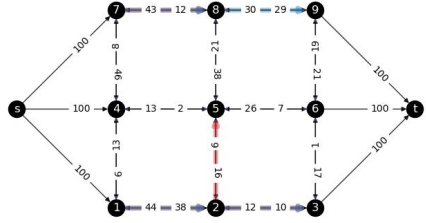
\begin{figure}
    \centering
    \begin{subfigure}[b]{0.55\textwidth}
        \centering
        \includegraphics[width=\textwidth]{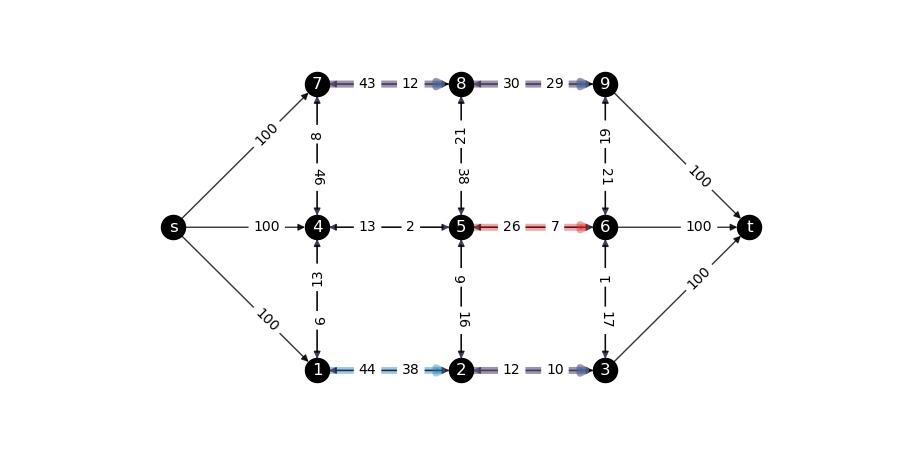}
        \caption{Defender incumbent and attacker's response.}
        \label{fig:solution-attack-9-2-0}
    \end{subfigure}%
    \begin{subfigure}[b]{0.4\textwidth}
        \centering
        \begin{tikzpicture}
  \draw
    (0, 0) node[draw=black,circle,fill=white] (0){0}
    (-2, -2) node[draw=black,circle,fill=red] (1){1}
    (2, -2) node[draw=black,circle,fill=red] (2){2};
  \begin{scope}[->]
    \draw[line width=1.5] (0) to node[] {$\xi_{78} = 1$} (1);
    \draw[line width=1.5] (0) to node[] {$\xi_{78} = 0$} (2);
  \end{scope}
\end{tikzpicture}
        \caption{Attacker refines partition tree.}
        \label{fig:tree-attacker-9-2-1}
    \end{subfigure}
    \\
    \begin{subfigure}[b]{0.5\textwidth}
        \centering
        \includegraphics[width=\textwidth]{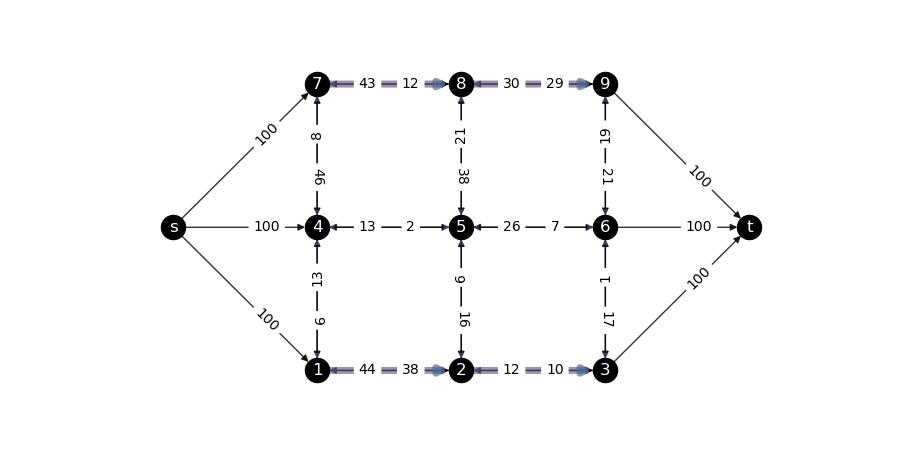}
        \caption{New incumbent solution for attacker.}
        \label{fig:solution-attack-9-3-1}
    \end{subfigure}
    \begin{subfigure}[b]{0.25\textwidth}
        \centering
        \scalebox{0.75}{\begin{tikzpicture}
  \draw
    (0, 0) node[draw=black,circle,fill=white] (0){0}
    (-2, -2) node[draw=black,circle,fill=white] (1){1}
    (-4, -4) node[draw=black,circle,fill=red] (3){3}
    (0, -4) node[draw=black,circle,fill=red] (4){4}
    (2, -2) node[draw=black,circle,fill=red] (2){2};
  \begin{scope}[->]
    \draw[line width=1.5] (0) to node[] {$\xi_{78} = 1$} (1);
    \draw[line width=1.5] (0) to node[] {$\xi_{78} = 0$} (2);
    \draw[line width=1.5] (1) to node[] {$\xi_{12} = 1$} (3);
    \draw[line width=1.5] (1) to node[] {$\xi_{12} = 0$} (4);
  \end{scope}
\end{tikzpicture}}
        \caption{Attacker refines partition tree again.}
        \label{fig:tree-attacker-9-3-2}
    \end{subfigure}
    \begin{subfigure}[b]{\textwidth}
        \centering
        \includegraphics[width=0.5\textwidth]{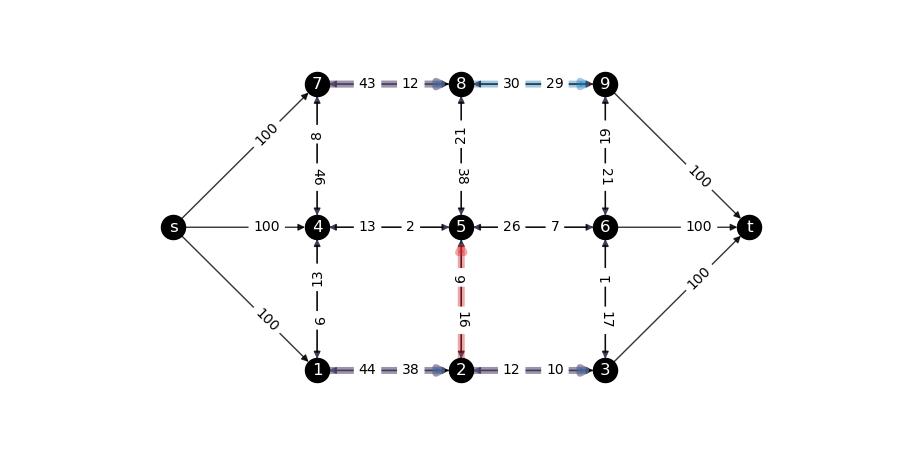}
        \caption{Final solution (expected maximum flow equals 10.5).}
        \label{solution-9-2}
    \end{subfigure}
    \caption{Final iteration of algorithm.}
    \label{fig:final-iteration}
\end{figure}

\section{Computational Results}\label{sec:results}
The SRA framework works well when the upper and lower bounds become sufficiently close before too many refinements are made to the partition. In this section we test this empirically. Specifically, we evaluate the computational efficiency of the SRA algorithm and compare it with a benchmark method, presenting results for both methods on the same expected maximum flow problem considered in \S\ref{sec:example}.

\subsection{Test Problem: Stochastic Maximum Flow Interdiction Problem with Hardening}\label{sec:test-prob}

We evaluate our approach using the defender-attacker problem with the maximum network flow problem as the recourse function. This problem involves three players: an attacker, a defender, and an operator. The operator aims to maximize the flow across a network $G = (N, A)$ from a source node $s$ to a target node $t$ under various failure scenarios. Let $\capac[k]$ represent the capacity of arc $k$, and let $\randvec$ be a random vector where $\randvark{k} = 1$ if arc $k$ is available and $0$ otherwise. The variable $y_k$ denotes the flow in arc $k = (i, j)$, and $y_{ts}$ represents the flow on the auxiliary arc $(t, s)$. Let $A^+_i = \{k = (i, j) \mid (i, j) \in A\}$ be the forward star set for node $i$, and $A^-_i = \{k = (j, i) \mid (j, i) \in A\}$ be the reverse star set. In this test problem we assumed that arcs emanating from the source or flowing into the sink cannot fail (to avoid easily disconnecting the network) and set the capacity of these arcs (as well as the reverse arc $(t,s)$) to a large value. The recourse function $\recfn(\randvec)$ represents the maximum flow given arc availability vector $\randvec$: 

\begin{maxi!}|s|
    {}{\secondstagevar_{ts} \label{eq:max-flow-obj}}{\label{mod:max-flow-recourse}}{\recfn(\randvec) =}
    \addConstraint{\sum_{k\in A^+_i} \secondstagevar_k - \sum_{k\in A^-_i} \secondstagevar_k}{= \delta(i)}{\quad \forall i\in N\label{flow-conserve}}
    \addConstraint{0 \leq \secondstagevar_\compon}{\leq \capac[k]\randvec_k}{\quad \forall k\in A, \label{max-flow-cap}}
\end{maxi!}

\noindent where $\delta(i)$ is an expression that equals $-y_{ts}$ for $i=s$, $y_{ts}$ for $i=t$, and 0 for all other $i\in N\setminus\{s,t\}$.

%Given this assumption, we now use the notation $\firststageallocvec=\left(\firststageallocil\right)_{i=1,\dots,\numcompon,\ell=0,\dots,\numalloclevels}$ to denote the binary version of the vector of first-stage decisions and the notation $\firststageallocfeasregion$ to denote its feasible region. (A similar representation is used for \attackvec.)
The defender has a budget $b^D$, with the feasible region defined as $\firststageallocfeasregion = \{\firststagevec \mid \sum_{k \in \arcs}\alloclevelssum \alloclevel \firststageallocil \leq \numalloclevels b^D,\alloclevelssum \firststageallocil \leq 1$ for all $\compon=1,\dots,\numcompon\}$.
The attacker’s decisions are constrained by a budget $b^A$, and the feasible region is $\attackfeasregion = \{\attackvec \mid \sum_{k \in \arcs}\attackalloclevelssum \attackalloclevel\attackallocil \leq \numattackalloclevels b^A,\alloclevelssum \attackallocil \leq 1$ for all $\compon=1,\dots,\numcompon\}$. We used \eqref{state-prob-fn} as the state probability function $f_k(\componoutcome;\alloclevel,\attackalloclevel)$.

%The objective function of the stochastic defender-attacker problem is given by:

% \begin{maxi}|s|
%     {\firststagevec\in \firststagefeasregion}{\min_{\attackvec\in \attackfeasregion}\quad  \sum_{\randvec\in \randvecsupport} \probability[\randvec \,|\, \firststagevec,\attackvec]\,\recfn(\randvec),}{\label{mod: DAD first-stage}}{}
% \end{maxi} 

For a given random vector $\randvec$, let $1(\randvec)$ be the set of arcs that are available and $0(\randvec)$ be the set of arcs that are not available. Then the probability that arc availability vector $\randvec$ is realized is:

\begin{equation}
    \probability[\randvec |\firststagevec, \attackvec] = \prod_{k \in 1(\randvec)} \alloclevelssum\attackalloclevelssum f_k(1; \alloclevel, \attackalloclevel) \firststagevar_{k, \alloclevel} \attackvar_{k, \attackalloclevel} \prod_{k \in 0(\randvec)} \alloclevelssum\attackalloclevelssum f_{k}(0; \alloclevel, \attackalloclevel) \firststagevar_{k, \alloclevel} \attackvar_{k, \attackalloclevel}
\end{equation}

%\subsubsection{Successive Refinement Algorithm}

The following convex equivalent recourse function, which is convex on $[0,1]^A$, exists for this problem \citep{Cormican.1998}. Note that because the attacker's problem is maximization, we need the equivalent recourse function to be \textit{convex}.
 
\begin{maxi!}|s|
    {}{\secondstagevar_{ts} - \sum_{k 
    \in \arcs}(1 - \randvark{k}) \secondstagevar_{k} \label{eq:pen-based-max-flow-obj}}{\label{mod:pen-based-max-flow-recourse}}{\recfnpenalty(\randvec) =}
    \addConstraint{\sum_{k\in A^+_i} \secondstagevar_k - \sum_{k\in A^-_i} \secondstagevar_k}{= \delta(i)}{\quad \forall i\in N\label{pen-based-flow-conserve}}
    \addConstraint{0 \leq \secondstagevar_\compon}{\leq \capac[k]}{\quad \forall k\in A, \label{pen-based-max-flow-cap}}
\end{maxi!}

For a fixed $\hat{\attackvec}$, the expected value for the random vector $\randvec$ is $\randvecevfn{\firststagevecfix,\attackvecfix} = \alloclevelssum\attackalloclevelssum f_k(1;\alloclevel, \attackalloclevel)\firststagevar_{k,\attackalloclevel} \hat{\attackvar}_{k,\alloclevel}$. Thus, the convex equivalent recourse problem for a given cell $\cell$ is given as:

\begin{maxi!}|s|
    {}{\secondstagevar_{ts} - \sum_{k 
    \in \arcs}\left(1 - \alloclevelssum\attackalloclevelssum f_k(1;\alloclevel, \attackalloclevel)\firststagevar_{k,\attackalloclevel} \hat{\attackvar}_{k,\alloclevel} \right) \secondstagevar_{k} \label{eq:mean-pen-based-max-flow-obj}}{\label{mod:mean-pen-based-max-flow-recourse}}{\recfnpenalty(\randvecevcellfnindex{\firststagevec,\attackvecfix}) = }
    \addConstraint{\sum_{k\in A^+_i} \secondstagevar_k - \sum_{k\in A^-_i} \secondstagevar_k}{= \delta(i)}{\quad \forall i\in N\label{mean-pen-based-flow-conserve}}
    \addConstraint{0 \leq \secondstagevar_\compon}{\leq \capac[k]}{\quad \forall k\in A, \label{mean-pen-based-max-flow-cap}}
\end{maxi!}

Consider a leaf node $\leaf$ in a partially refined partition $\partition$ where arcs in the subset $\fixedcomponents$ have fixed states (i.e., are either in $1(\leaf)$ and $0(\leaf)$) and $\components \setminus \fixedcomponents$ represents the subset of arcs that are unfixed and assume their expected state.  Let $\probfnvar^{\leaf',\leaf} [\firststagevecfix, \attackvec] =\prod_{k \in 1(\treenode')} \alloclevelssum\attackalloclevelssum f_k(1; \alloclevel, \attackalloclevel) \hat{\firststagevar}_{k, \alloclevel} \attackvar_{k, \attackalloclevel} \prod_{k \in 0(\treenode')} \alloclevelssum\attackalloclevelssum f_k(0; \alloclevel, \attackalloclevel) \hat{\firststagevar}_{k, \alloclevel} \attackvar_{k, \attackalloclevel}$. For a fixed defense (hardening) plan $\hat{\firststagevec}$, as explained in \S\ref{sec:app-binary} the attacker's problem is formulated as:

\begin{mini!}|s|
    {\attackvec \in \attackfeasregion}{\auxvar_0 \label{eq:max-prob-epi-obj}}{\label{mod:max-prob-epi}}{}
    \addConstraint{\sum_{\treenode'\in \descendentnodes(\treenode;\partition)} \Bigl(\probfnvar^{\leaf',\leaf} [\firststagevecfix, \attackvec]\Bigl) \auxvar_{\treenode'}}{\geq \auxvar_\treenode}{\quad \forall \treenode\in \parttreenodes(\partition)\setminus \parttreeleafnodes(\partition)\label{eqn:conditional}}
    \addConstraint{\sum_{k \in 1(\treenode)} \capac[k]\capdualvar_k^\leaf}{\geq \auxvar_{\treenode}}{\quad \forall \treenode \in \parttreenodes(\partition) \setminus \parttreeleafnodes(\partition)}
    \addConstraint{\constrdualvar_i^\leaf - \constrdualvar_j^\leaf + \capdualvar_k^\leaf + ( 1 - \alloclevelssum\attackalloclevelssum f_k(0;\alloclevel,\attackalloclevel) \hat{\firststagevar}_{k\alloclevel}\attackvar_{k\attackalloclevel})}{\geq 0}{\quad \forall k \in \components \setminus \fixedcomponents(\leaf), \leaf \in \parttreeleafnodes(\partition)}
    \addConstraint{\constrdualvar_{i}^{\leaf} - \constrdualvar_{j}^{\leaf} + \capdualvar_k^{\leaf} + 1}{\geq 0}{\quad \forall \compon=(i,j)\in 0(\leaf),\leaf \in \parttreeleafnodes(\partition)}
    \addConstraint{\constrdualvar_{i}^{\leaf} - \constrdualvar_{j}^{\leaf} + \capdualvar_k^{\leaf} }{\geq 0}{\quad \forall \compon=(i,j)\in 1(\leaf),\leaf \in \parttreeleafnodes(\partition)}
    \addConstraint{\constrdualvar_{s}^{\leaf}}{=}{0,\quad\leaf \in \parttreeleafnodes(\partition)}
    \addConstraint{\constrdualvar_{t}^{\leaf}}{=}{1,\quad\leaf \in \parttreeleafnodes(\partition)}
\end{mini!}

\noindent where $\constrdualvec$  and $\constrdualubvec$ are the dual vectors of constraints \eqref{pen-based-flow-conserve} and \eqref{pen-based-max-flow-cap}. Similarly, the defender's problem is formulated using formulation \eqref{mod:epi-defender-benders}:

\begin{mini!}|s|
    {\firststagevec\in\firststagefeasregion}{\auxvar_0 \label{eq:epi-obj-defender-benders-example}}{\label{mod:epi-defender-benders-example}}{}
    \addConstraint{\auxvar_0}{\geq \auxvar_{0,\attackplan}}{\quad\attackplansinlist\label{eqn:conditional-epigraph-root-benders-example}}
    \addConstraint{\auxvar_{\treenode,\attackplan}}{\geq \sum_{\treenode'\in \descendentnodes(\treenode;\partition)} \condtreeprob{\firststagevec,\attackvec^\attackplan}\,\auxvar_{\treenode',\attackplan}}{\quad \forall \treenode\in \parttreenodes(\partition)\setminus \parttreeleafnodes(\partition),\attackplansinlist\label{eqn:conditional-epigraph-benders-example}}
    \addConstraint{\auxvar_{\leaf,\attackplan}}{\geq \sum_{k \in 1(\treenode)} \capac[k]\hat{\capdualvar}_k^{\leaf\attackplan}}{\quad\forall \leaf\in \parttreeleafnodes(\partition),\attackplansinlist\label{eqn:leaf-cut-benders-example}}
\end{mini!}

\subsection{Results}
We compared the SRA mthod with the deterministic equivalent formulation (DEF) benchmark from \S\ref{sec:app-binary}. Experiments were conducted on grid networks of varying sizes (see Figure (\ref{fig:solution-attack-4-3-0} for example), with equal number of rows and columns. These networks include bidirectional arcs where the source node connects to all nodes in the leftmost column, and the target node connects to all nodes in the rightmost column. We vary the defender's and attacker's budgets, as well as the number of allocation levels, creating unique experimental instances defined by grid size (rows × columns), attacker/defender budget ($b^A=b^D$), and the number of allocation levels for both the defender and attacker ($\numalloclevels=\numattackalloclevels$). The maximum budget available to the defender and attacker is calculated as the product of the budget and allocation i.e., the defender budget is $b^{D}\numalloclevels$ and the attacker's $b^{A}\numattackalloclevels$. The grid sizes considered in our experiments range from $3\times3$ (9 nodes) to $9\times9$ (81 nodes), and the budgets and allocation levels range between 2 and 4. Because Gurobi does not allow adding variables during the branch-and-cut algorithm, we added all necessary variables at the beginning. We set a time limit of $1800s$ for $3 \times 3$ networks, $3600s$ for  $4 \times 4$ networks, and $7200s$ for others.

We conducted five randomly generated sample instances for each problem instance. Table \ref{tab:results1} presents the results obtained using both the DEF and the SRA approaches. The table lists the runtimes when the model solved within the time limit and the relative optimality gap otherwise. Additionally, we tracked the number of instances (out of five) that solved within the time limit. For the SRA method, we also recorded the number of average refinements required in Algorithm \ref{alg:sra-callback}.

As shown in Table \ref{tab:results1}, the DEF's effectiveness is limited to smaller problems with a budget of 2. This limitation is due to the high number of arcs, which results in a large number of multilinear terms with high degree in constraints \eqref{constr:epi-lb} and \eqref{eq:obj-multi-attack}. In contrast, the SRA method achieved optimal solutions for most instances, up to $9 \times 9$ grid networks with 81 nodes, particularly when the budget was either 2 or 3 units, and there were 2 or 3 allocation levels. The SRA method performed between 3.6 and 67 times faster for instances where the DEF method could solve. 

The findings consistently show that as the network size, budget, or allocation levels increase, the SRA algorithm requires more computation time to reach solutions. The increased number of arcs in larger networks adds to the problem's complexity and nonlinearity. This increased complexity results in a deeper refinement tree for SRA, necessitating more extensive refinement efforts to close the optimality gap. Higher budgets further intensify this effect, deviating more from the optimal objective of the mean value problem and requiring additional refinements to reduce the gap.


% \usepackage{longtable}
% Note: It may be necessary to compile the document several times to get a multi-page table to line up properly
\begin{longtable}{c|c|c|cc|ccc}
\caption{Results}\label{tab:results1}\\
\hline
\multirow{2}{*}{\textbf{rows $\times$ columns}} &
  \multirow{2}{*}{\textbf{$b^{A} = b^{D}$}} &
  \multirow{2}{*}{\textbf{$L = L'$}} &
  \multicolumn{2}{c|}{\textbf{DEF}} &
  \multicolumn{3}{c}{\textbf{SRA}} \\ \cline{4-8} 
  & & & \multicolumn{1}{c}{\textbf{Runtime}} & \multicolumn{1}{c|}{\textbf{Solved}} & \multicolumn{1}{c}{\textbf{Runtime}} & \multicolumn{1}{c}{\textbf{No of Refinements}} & \multicolumn{1}{c}{\textbf{Solved}} \\ \hline
\endhead
\hline
\endfoot
\endlastfoot
3  & 2 & 2      & 11.7   & 5       & 3.2    & 3.6  & 5 \\
3  & 2 & 3      & 193.3  & 5       & 14.8   & 7.4  & 5 \\
3  & 2 & 4      &        & 0       & 64.3   & 16.4 & 5 \\
3  & 3 & 2      &        & 0       & 14.0   & 13.0 & 5 \\
3  & 3 & 3      &        & 0       & 85.9   & 28.8 & 5 \\
3  & 3 & 4      &        & 0       & 253.0  & 26.2 & 5 \\
3  & 4 & 2      &        & 0       & 65.3   & 31.0 & 5 \\
3  & 4 & 3      &        & 0       & 218.0  & 45.8 & 5 \\
3  & 4 & 4      &        & 0       & 706.2  & 36.5 & 2 \\ \hline
4  & 2 & 2      & 65.2   & 5       & 4.7    & 1.2  & 5 \\
4  & 2 & 3      & 3461.9 & 1       & 59.3   & 12.0 & 5 \\
4  & 2 & 4      &        & 0       & 597.4  & 32.2 & 5 \\
4  & 3 & 2      &        & 0       & 39.6   & 17.2 & 5 \\
4  & 3 & 3      &        & 0       & 1700.4 & 55.8 & 5 \\
4  & 3 & 4      &        & 0       & 4.7\%  & 46.0 & 0 \\
4  & 4 & 2      &        & 0       & 322.1  & 59.0 & 5 \\
4  & 4 & 3      &        & 0       & 5.9\%  & 76.0 & 0 \\
4  & 4 & 4      &        & 0       &        & 11.2 & 0 \\ \hline
5  & 2 & 2      & 119.5  & 5       & 7.0    & 0.0  & 5 \\
5  & 2 & 3      &        & 0       & 68.4   & 4.6  & 5 \\
5  & 2 & 4      &        & 0       & 1068.8 & 27.8 & 5 \\
5  & 3 & 2      &        & 0       & 42.0   & 5.8  & 5 \\
5  & 3 & 3      &        & 0       & 3439.1 & 50.2 & 5 \\
5  & 3 & 4      &        & 0       & 12.6\% & 23.8 & 0 \\
5  & 4 & 2      &        & 0       & 474.0  & 39.4 & 5 \\
5  & 4 & 3      &        & 0       &        & 21.4 & 0 \\
5  & 4 & 4      &        & 0       &        & 5.6  & 0 \\ \hline
6  & 2 & 2      & 163.9  & 5       & 6.6    & 0.0  & 5 \\
6  & 2 & 3      &        & 0       & 56.5   & 0.6  & 5 \\
6  & 2 & 4      &        & 0       & 1175.5 & 9.8  & 5 \\
6  & 3 & 2      &        & 0       & 42.1   & 1.8  & 5 \\
6  & 3 & 3      &        & 0       & 4553.2 & 38.0 & 4 \\
6  & 3 & 4      &        & 0       & 14.7\% & 13.0 & 0 \\
6  & 4 & 2      &        & 0       & 453.1  & 20.4 & 5 \\
6  & 4 & 3      &        & 0       & 18.3\% & 15.4 & 0 \\
6  & 4 & 4      &        & 0       &        & 0.4  & 0 \\ \hline
7  & 2 & 2      & 353.0  & 5       & 10.6   & 0.2  & 5 \\
7  & 2 & 3      &        & 0       & 70.1   & 0.0  & 5 \\
7  & 2 & 4      &        & 0       & 602.9  & 2.6  & 5 \\
7  & 3 & 2      &        & 0       & 40.8   & 0.0  & 5 \\
7  & 3 & 3      &        & 0       & 2064.4 & 13.8 & 5 \\
7  & 3 & 4      &        & 0       &        & 3.2  & 0 \\
7  & 4 & 2      &        & 0       & 333.1  & 6.8  & 5 \\
7  & 4 & 3      &        & 0       &        & 8.0  & 0 \\
7  & 4 & 4      &        & 0       &        & 0.8  & 0 \\ \hline
8  & 2 & 2      & 662.7  & 5       & 17.1   & 0.0  & 5 \\
8  & 2 & 3      &        & 0       & 52.8   & 0.0  & 5 \\
8  & 2 & 4      &        & 0       & 454.3  & 0.4  & 5 \\
8  & 3 & 2      &        & 0       & 36.8   & 0.0  & 5 \\
8  & 3 & 3      &        & 0       & 1779.1 & 4.2  & 5 \\
8  & 3 & 4      &        & 0       &        & 4.2  & 0 \\
8  & 4 & 2      &        & 0       & 189.0  & 0.4  & 5 \\
8  & 4 & 3      &        & 0       &        & 4.4  & 0 \\
8  & 4 & 4      &        & 0       &        & 0.0  & 0 \\ \hline
9  & 2 & 2      & 1238.9 & 5       & 18.5   & 0.2  & 5 \\
9  & 2 & 3      &        & 0       & 63.5   & 0.2  & 5 \\
9  & 2 & 4      &        & 0       & 648.6  & 0.8  & 5 \\
9  & 3 & 2      &        & 0       & 48.5   & 0.2  & 5 \\
9  & 3 & 3      &        & 0       &        &      & 0 \\
9  & 3 & 4      &        & 0       &        & 5.0  & 0 \\
9  & 4 & 2      &        & 0       & 324.8  & 0.8  & 5 \\
9  & 4 & 3      &        & 0       &        & 4.2  & 0 \\
9  & 4 & 4      &        & 0       &        & 0.0  & 0 \\ \hline
\end{longtable}

\section{Conclusion}\label{sec:conclusion}

This study presented an approach for solving tri-level stochastic defender-attacker problems with decision-dependent probability distributions. This class of problems addresses both imperfect and probabilistic hardening and interdiction in which the probability distribution of the capacity of a system component depends on the defender's and attacker's allocation. The large amount of high-degree multilinear terms in the deterministic equivalent formulation make solving these problems particularly challenging. 

To address these challenges, we developed a successive refinement framework for tri-level defender-attacker problems, leveraging the lower and upper-bound properties inherent in these problems. We then demonstrated how to apply this framework to a specific class of tri-level stochastic defender-attacker problems with decision-dependent capacity distributions. 

We then tested the algorithm on a set of instances of the tri-level defender-attacker maximum flow problem, comparing the successive refinement framework's performance with the deterministic equivalent formulation (DEF). For each parameter setting, we ran both models for five randomly generated instances. The Gurobi solver solved the DEF formulation for all five instances for only the smallest problem instances tested. For the parameter settings in which the DEF method was able to solve any instances, the SRA method was between $3.6$ and $67$ times faster. Also, the SRA framework was able to solve instances consisting of $9\times 9$ grid networks with 81 nodes with up to four allocation levels and a budget of eight units for both the defender and attacker. The findings indicate that the SRA approach scales well with increasing network size and complexity, obtaining solutions in most cases under various attacker and defender budget and allocation scenarios. However, the SRA method did reach its limits on the largest instances tested. It was not able to solve all five instances for any of the networks for the largest budget size (16 allocation units).

In applying the successive refinement framework to tri-level interdiction problems with hardening, it was assumed that the random variables governing uncertain capacity distributions were mutually independent and discrete with finite support. Thus, future work could examine how to apply the framework to problems with correlated random variables, if a strong motivating application exists. Future work could also examine how to apply the framework to problems with an infinite support, either discrete or continuous.

\bibliographystyle{plainnat}  
\bibliography{BiTri,ddu}  

\begin{thebibliography}{28}
\providecommand{\natexlab}[1]{#1}
\providecommand{\url}[1]{\texttt{#1}}
\expandafter\ifx\csname urlstyle\endcsname\relax
  \providecommand{\doi}[1]{doi: #1}\else
  \providecommand{\doi}{doi: \begingroup \urlstyle{rm}\Url}\fi

\bibitem[Aksen et~al.(2014)Aksen, Akca, and Aras]{aksen2014bilevel}
Deniz Aksen, Sema~{\c{S}}eng{\"u}l Akca, and Necati Aras.
\newblock A bilevel partial interdiction problem with capacitated facilities and demand outsourcing.
\newblock \emph{Computers \& Operations Research}, 41:\penalty0 346--358, 2014.

\bibitem[Cormican et~al.(1998)Cormican, Morton, and Wood]{Cormican.1998}
Kelly~J Cormican, David~P Morton, and R~Kevin Wood.
\newblock {Stochastic Network Interdiction}.
\newblock \emph{Operations Research}, 46\penalty0 (2):\penalty0 184--197, 1998.
\newblock ISSN 0030-364X.
\newblock \doi{10.1287/opre.46.2.184}.

\bibitem[Du and Peeta(2014)]{Du.Peeta.2014}
Lili Du and Srinivas Peeta.
\newblock {A Stochastic Optimization Model to Reduce Expected Post-Disaster Response Time Through Pre-Disaster Investment Decisions}.
\newblock \emph{Networks and Spatial Economics}, 14\penalty0 (2):\penalty0 271--295, 2014.
\newblock ISSN 1566-113X.
\newblock \doi{10.1007/s11067-013-9219-1}.

\bibitem[Fard and Hajiaghaei-Keshteli(2018)]{fard2018bi}
Amir Mohammad~Fathollahi Fard and Mostafa Hajiaghaei-Keshteli.
\newblock A bi-objective partial interdiction problem considering different defensive systems with capacity expansion of facilities under imminent attacks.
\newblock \emph{Applied Soft Computing}, 68:\penalty0 343--359, 2018.

\bibitem[Goel and Grossmann(2006)]{Goel.Grossmann.2006}
Vikas Goel and Ignacio~E. Grossmann.
\newblock {A Class of stochastic programs with decision dependent uncertainty}.
\newblock \emph{Mathematical Programming}, 108\penalty0 (2-3):\penalty0 355--394, 2006.
\newblock ISSN 0025-5610.
\newblock \doi{10.1007/s10107-006-0715-7}.

\bibitem[Goel et~al.(2006)Goel, Grossmann, El-Bakry, and Mulkay]{Goel.Mulkay.2006}
Vikas Goel, Ignacio~E. Grossmann, Amr~S. El-Bakry, and Eric~L. Mulkay.
\newblock {A novel branch and bound algorithm for optimal development of gas fields under uncertainty in reserves}.
\newblock \emph{Computers \& Chemical Engineering}, 30\penalty0 (6-7):\penalty0 1076--1092, 2006.
\newblock ISSN 0098-1354.
\newblock \doi{10.1016/j.compchemeng.2006.02.006}.

\bibitem[Hausken and Bier(2011)]{hausken2011defending}
Kjell Hausken and Vicki~M Bier.
\newblock Defending against multiple different attackers.
\newblock \emph{European Journal of Operational Research}, 211\penalty0 (2):\penalty0 370--384, 2011.

\bibitem[Held and Woodruff(2005)]{Held.Woodruff.2005}
Harald Held and David~L. Woodruff.
\newblock {Heuristics for Multi-Stage Interdiction of Stochastic Networks}.
\newblock \emph{Journal of Heuristics}, 11\penalty0 (5-6):\penalty0 483--500, 2005.
\newblock ISSN 1381-1231.
\newblock \doi{10.1007/s10732-005-3122-y}.

\bibitem[Hien et~al.(2020)Hien, Sim, and Xu]{hien2020mitigating}
Le~Thi~Khanh Hien, Melvyn Sim, and Huan Xu.
\newblock Mitigating interdiction risk with fortification.
\newblock \emph{Operations Research}, 68\penalty0 (2):\penalty0 348--362, 2020.

\bibitem[Huang et~al.(1977)Huang, Ziemba, and Ben-Tal]{huang1977bounds}
CC~Huang, William~T Ziemba, and Aharon Ben-Tal.
\newblock Bounds on the expectation of a convex function of a random variable: With applications to stochastic programming.
\newblock \emph{Operations Research}, 25\penalty0 (2):\penalty0 315--325, 1977.

\bibitem[Hunt and Zhuang(2024)]{hunt2024review}
Kyle Hunt and Jun Zhuang.
\newblock A review of attacker-defender games: Current state and paths forward.
\newblock \emph{European Journal of Operational Research}, 313\penalty0 (2):\penalty0 401--417, 2024.

\bibitem[Karaesmen and Van~Ryzin(2004)]{karaesmen2004overbooking}
Itir Karaesmen and Garrett Van~Ryzin.
\newblock Overbooking with substitutable inventory classes.
\newblock \emph{Operations Research}, 52\penalty0 (1):\penalty0 83--104, 2004.

\bibitem[Li et~al.(2021)Li, Li, Zhang, and Gan]{Li.2021p8}
Qing Li, Mingchu Li, Runfa Zhang, and Jianyuan Gan.
\newblock {A stochastic bilevel model for facility location-protection problem with the most likely interdiction strategy}.
\newblock \emph{Reliability Engineering \& System Safety}, 216:\penalty0 108005, 2021.
\newblock ISSN 0951-8320.
\newblock \doi{10.1016/j.ress.2021.108005}.

\bibitem[Liu et~al.(2022)Liu, Li, and Sen]{Liu.Sen.2022}
Junyi Liu, Guangyu Li, and Suvrajeet Sen.
\newblock {Coupled Learning Enabled Stochastic Programming with Endogenous Uncertainty}.
\newblock \emph{Mathematics of Operations Research}, 47\penalty0 (2):\penalty0 1681--1705, 2022.
\newblock ISSN 0364-765X.
\newblock \doi{10.1287/moor.2021.1185}.

\bibitem[Losada et~al.(2012)Losada, Scaparra, Church, and Daskin]{Losada.2012}
Chaya Losada, M.~Paola Scaparra, Richard~L. Church, and Mark~S. Daskin.
\newblock {The stochastic interdiction median problem with disruption intensity levels}.
\newblock \emph{Annals of Operations Research}, 201\penalty0 (1):\penalty0 345--365, 2012.
\newblock ISSN 0254-5330.
\newblock \doi{10.1007/s10479-012-1170-x}.

\bibitem[Luo and Mehrotra(2020)]{Luo.Mehrotra.2020}
Fengqiao Luo and Sanjay Mehrotra.
\newblock {Distributionally robust optimization with decision dependent ambiguity sets}.
\newblock \emph{Optimization Letters}, 14\penalty0 (8):\penalty0 2565--2594, 2020.
\newblock ISSN 1862-4472.
\newblock \doi{10.1007/s11590-020-01574-3}.

\bibitem[Ma et~al.(2017)Ma, Su, Wang, Qiu, and Guo]{Ma.Guo.2017}
Shanshan Ma, Liu Su, Zhaoyu Wang, Feng Qiu, and Ge~Guo.
\newblock {Resilience Enhancement of Distribution Grids Against Extreme Weather Events}.
\newblock \emph{IEEE Transactions on Power Systems}, 33\penalty0 (5):\penalty0 4842--4853, 2017.
\newblock ISSN 0885-8950.
\newblock \doi{10.1109/tpwrs.2018.2822295}.

\bibitem[McCormick(1976)]{mccormick1976computability}
Garth~P McCormick.
\newblock Computability of global solutions to factorable nonconvex programs: Part i—convex underestimating problems.
\newblock \emph{Mathematical programming}, 10\penalty0 (1):\penalty0 147--175, 1976.

\bibitem[McMasters and Mustin(1970)]{mcmasters1970optimal}
Alan~W McMasters and Thomas~M Mustin.
\newblock Optimal interdiction of a supply network.
\newblock \emph{Naval Research Logistics Quarterly}, 17\penalty0 (3):\penalty0 261--268, 1970.

\bibitem[Medal and Affar(2024)]{medalaffar24}
Hugh Medal and Samuel Affar.
\newblock {A Successive Refinement for Solving Stochastic Programs with Decision-Dependent Random Capacities}.
\newblock Technical report, University of Tennessee-Knoxville, 2024.
\newblock URL \url{https://medalgroup.org/wp-content/uploads/2024/04/MedalAffar2024.pdf}.

\bibitem[Medal et~al.(2015)Medal, Pohl, and Rossetti]{Medal.2015}
Hugh~R. Medal, Edward~A. Pohl, and Manuel~D. Rossetti.
\newblock {Allocating Protection Resources to Facilities When the Effect of Protection is Uncertain}.
\newblock \emph{IIE Transactions}, 48\penalty0 (3):\penalty0 220--234, 2015.
\newblock ISSN 0740-817X.
\newblock \doi{10.1080/0740817x.2015.1078013}.

\bibitem[O’Hanley et~al.(2013)O’Hanley, Scaparra, and García]{OHanley.2013}
Jesse~R. O’Hanley, M.~Paola Scaparra, and Sergio García.
\newblock {Probability chains: A general linearization technique for modeling reliability in facility location and related problems}.
\newblock \emph{European Journal of Operational Research}, 230\penalty0 (1):\penalty0 63--75, 2013.
\newblock ISSN 0377-2217.
\newblock \doi{10.1016/j.ejor.2013.03.021}.

\bibitem[Peeta et~al.(2010)Peeta, Salman, Gunnec, and Viswanath]{Peeta.Viswanath.2010}
Srinivas Peeta, F.~Sibel Salman, Dilek Gunnec, and Kannan Viswanath.
\newblock {Pre-disaster investment decisions for strengthening a highway network}.
\newblock \emph{Computers \& Operations Research}, 37\penalty0 (10):\penalty0 1708--1719, 2010.
\newblock ISSN 0305-0548.
\newblock \doi{10.1016/j.cor.2009.12.006}.

\bibitem[Smith and Song(2020)]{Smith.2020}
J.~Cole Smith and Yongjia Song.
\newblock {A survey of network interdiction models and algorithms}.
\newblock \emph{European Journal of Operational Research}, 283\penalty0 (3):\penalty0 797--811, 2020.
\newblock ISSN 0377-2217.
\newblock \doi{10.1016/j.ejor.2019.06.024}.

\bibitem[Wood(1993)]{Wood.1993}
R.Kevin Wood.
\newblock {Deterministic network interdiction}.
\newblock \emph{Mathematical and Computer Modelling}, 17\penalty0 (2):\penalty0 1--18, 1993.
\newblock ISSN 0895-7177.
\newblock \doi{10.1016/0895-7177(93)90236-r}.

\bibitem[Xu and Zhuang(2016)]{xu2016modeling}
Jie Xu and Jun Zhuang.
\newblock Modeling costly learning and counter-learning in a defender-attacker game with private defender information.
\newblock \emph{Annals of Operations Research}, 236\penalty0 (1):\penalty0 271--289, 2016.

\bibitem[Yin et~al.(2023)Yin, Feng, and Hou]{Yin.Hou.2023}
Wenqian Yin, Shuanglei Feng, and Yunhe Hou.
\newblock {Stochastic Wind Farm Expansion Planning With Decision-Dependent Uncertainty Under Spatial Smoothing Effect}.
\newblock \emph{IEEE Transactions on Power Systems}, 38\penalty0 (3):\penalty0 2845--2857, 2023.
\newblock ISSN 0885-8950.
\newblock \doi{10.1109/tpwrs.2022.3184705}.

\bibitem[Zhu et~al.(2013)Zhu, Zheng, Zhang, and Cai]{Zhu.2013}
Yueni Zhu, Zheng Zheng, Xiaoyi Zhang, and Kaiyuan Cai.
\newblock {The r-interdiction median problem with probabilistic protection and its solution algorithm}.
\newblock \emph{Computers \& Operations Research}, 40\penalty0 (1):\penalty0 451--462, 2013.
\newblock ISSN 0305-0548.
\newblock \doi{10.1016/j.cor.2012.07.017}.

\end{thebibliography}

\newpage
\appendix

\section{Non-Callback Algorithms for Attacker's Problem and Defender-Attacker Problem}

\begin{algorithm}
	{\sc SuccessiveRefinementAlgorithm}($\firststagevecfix,\epsilon$)
	\begin{algorithmic}[1]
        \STATE $\partition\gets\{\randvecsupport\}$, $K = 0$, $e^{max}\gets +\infty$.
		\WHILE{$e^{max} > \epsilon$}
        \STATE Solve \eqref{mod:attacker} for partition $\partition$, returning solution $\attackvec$.
        \STATE Let $e^{max} = \max_{\partitioncell\in \partition} \{\textsc{Error}(\partitioncell,\firststagevecfix,\attackvec)\}$ and $\partitioncellvar^{max}\in \arg\max_{\partitioncell\in \partition} \{\textsc{Error}(\partitioncell,\firststagevecfix,\attackvec)\}$ (see \eqref{eqn:approx-error-cell}).\label{line:compute-error}
        \IF{$e^{max} > \epsilon$}
        \STATE $\partition\gets \textsc{Refine}(\partition,\firststagevecfix,\attackvec,\partitioncellvar^{max})$ (see \S\ref{sec:refine-cell}). \label{line:refine}
        \STATE Modify formulation \eqref{mod:attacker} according to new partition $\partition$.
        \ENDIF
		\ENDWHILE
		\RETURN $UB$, $x^*$
	\end{algorithmic}
	\caption{Successive refinement algorithm for attacker's problem without callback.}
	\label{alg:sra-attacker}
\end{algorithm}

\begin{algorithm}
	{\sc SuccessiveRefinementAlgorithm}($\epsilon$)
	\begin{algorithmic}[1]
        \STATE $\partition\gets\{\randvecsupport\}$, $K = 0$, $UB\gets +\infty$.
		\WHILE{$UB - LB > \epsilon\, LB$}
        \STATE Solve \eqref{mod:defender-epi-resticted} for partition $\partition$, returning solution $(\firststagevecfix,\auxvar)$. Set $LB=\auxvar$.
        \STATE With $\firststagevecfix$ fixed, solve \eqref{mod:attacker} using Algorithm \ref{alg:sra-attacker}, returning attacker's best response $\attackvec^*$ and objective value $\attackobjvar^*(\firststagevecfix)$. Set $UB = \attackobjvar^*(\firststagevecfix)$, $K=K+1$, and $\attackvec^K = \attackvec^*$.
        \STATE Add to \eqref{mod:defender-epi-resticted} a new cut of the form \eqref{eqn:attack-cut-epi-restricted}.
        \STATE Let $e^{max} = \max_{\partitioncell\in \partition} \{\textsc{Error}(\partitioncell,\firststagevecfix,\attackvec^*)\}$ and $\partitioncellvar^{max}\in \arg\max_{\partitioncell\in \partition} \{\textsc{Error}(\partitioncell,\firststagevecfix,\attackvec^*)\}$ (see \eqref{eqn:approx-error-cell}).
        \IF{$e^{max} > \epsilon$}
        \STATE $\partition\gets \textsc{Refine}(\partition,\firststagevecfix,\attackvec^*,\partitioncellvar^{max})$ (see \S\ref{sec:refine-cell}).
        \STATE Modify formulation \eqref{mod:defender-epi-resticted} according to new partition $\partition$.
        \ENDIF
		\ENDWHILE
		\RETURN $UB$, $x^*$
	\end{algorithmic}
	\caption{Successive refinement algorithm for defender-attacker without callback.}
	\label{alg:sra-da}
\end{algorithm}

% \section{Epigraph Formulation of Defender Attacker Problem with Primal Recourse Variables for Each Leaf}
% \begin{mini!}|s|
%     {\firststagevec\in\firststagefeasregion}{\auxvar_0 \label{eq:epi-obj-defender}}{\label{mod:epi-defender}}{}
%     \addConstraint{\auxvar_0}{\geq \auxvar_{0,\attackplan}}{\quad\attackplansinlist\label{eqn:conditional-epigraph-root}}
%     \addConstraint{\auxvar_{\treenode,\attackplan}}{\geq \sum_{\treenode'\in \descendentnodes(\treenode;\partition)} \condtreeprob{\firststagevec,\attackvec^\attackplan}\,\auxvar_{\treenode',\attackplan}}{\quad \forall \treenode\in \parttreenodes(\partition)\setminus \parttreeleafnodes(\partition),\attackplansinlist\label{eqn:conditional-epigraph}}
%     \addConstraint{\auxvar_{\leaf,\attackplan}}{\geq \secondstagecostvec^T\secondstagevec^{\cell,\attackplan}}{\quad\forall \leaf\in \parttreeleafnodes(\partition),\attackplansinlist\label{eqn:leaf-cut}}
%     \addConstraint{\secondstagematrixfixed\secondstagevec^{\cell,\attackplan}}{= \secondstagerhsfixedvec}{\quad \forall \cell=1,\dots,\numcells,\attackplansinlist}
%     \addConstraint{\secondstagevec^{\cell,\attackplan}}{\leq \secondstagerandvarcoefvec^T\,\randvecevcellfnindex{\firststagevec,\attackvec^\attackplan}}{\quad \forall \cell=1,\dots,\numcells,\attackplansinlist\label{eqn:leaf-ub}.}
% \end{mini!}

\end{document}